\documentclass{amsart}
\usepackage{amsmath,amsfonts,amsthm,amsrefs,amssymb,tikz-cd,paralist}
\usepackage{hyperref}

\newcommand{\Z}{\mathbb{Z}}
\newcommand{\Q}{\mathbb{Q}}
\newcommand{\R}{\mathbb{R}}
\newcommand{\C}{\mathbb{C}}

\newcommand{\calF}{\mathcal{F}}

\newcommand{\calO}{\mathcal{O}}

\newcommand{\End}{\text{End}}

\newcommand{\supp}{\text{supp}}

\newcommand{\cspn}{\overline{\text{span}}}
\newcommand{\dist}{\text{dist}}

\newcommand{\nctimes}{\times^\text{nc}}

\newtheorem{thm}{Theorem}[section]
\newtheorem{cor}[thm]{Corollary}
\newtheorem{prop}[thm]{Proposition}
\newtheorem{lem}[thm]{Lemma}

\theoremstyle{definition}
\newtheorem{rem}[thm]{Remark}

\newtheorem{defn}[thm]{Definition}

\newtheorem{eg}[thm]{Example}

\author{Adam Humeniuk}
\address{Adam Humeniuk,
Department of Pure Mathematics,
University of Waterloo, 
200 University Avenue West, 
Waterloo, Ontario, Canada 
N2L 3G1 }
\email{adam.humeniuk@uwaterloo.ca}
\thanks{Author supported by NSERC Alexander Graham Bell Canada Graduate Scholarship-Doctoral.}

\title[C*-envelopes of semicrossed products]{C*-envelopes of semicrossed products by lattice ordered abelian semigroups}
\date{\today}

\begin{document}

\begin{abstract}
A semicrossed product is a non-selfadjoint operator algebra encoding the action of a semigroup on an operator or C*-algebra. We prove that, when the positive cone of a discrete lattice ordered abelian group acts on a C*-algebra, the C*-envelope of the associated semicrossed product is a full corner of a crossed product by the whole group. By constructing a C*-cover that itself is a full corner of a crossed product, and computing the Shilov ideal, we obtain an explicit description of the C*-envelope. This generalizes a result of Davidson, Fuller, and Kakariadis from $\mathbb{Z}_+^n$ to the class of all discrete lattice ordered abelian groups.
\end{abstract}

\subjclass[2010]{47L25, 47L55, 47L65}
\keywords{dynamical systems of operator algebras, crossed product, semicrossed product, Nica-covariant, C*-envelope}

\maketitle

\section{Introduction} \label{sec:intro}

\subsection{Preliminaries}

A semicrossed product is a non-selfadjoint generalization of the crossed product of a C*-algebra by a group. A crossed product $B\rtimes G$ encodes the action of a group $G$ on a C*-algebra $B$, by embedding both into a larger C*-algebra in which the $G$-action is by unitaries. Built similarly, a semicrossed product of a (possibly non-selfadjoint) operator algebra $A$ by an abelian semigroup $P$ encodes a given action of $P$ on $A$ by completely contractive endomorphisms. First introduced by Arveson in \cite{arveson_operator_1967}, and first formally studied by Peters in \cite{peters_semi-crossed_1984} in the case $P=\Z_+$, subsequent work on semicrossed products has focused on conjugacy problems \cites{arveson_operator_1969,davidson_conjugate_2014,davidson_isomorphisms_2008,davidson_operator_2011,davidson_noncommutative_1998,hadwin_operator_1988,kakariadis_representations_2014} and their C*-envelopes \cites{davidson_semicrossed_2017,davidson_dilating_2010,kakariadis_semicrossed_2011,kakariadis_semicrossed_2012,katsoulis_tensor_2006,muhly_tensor_1998}. For a complete survey of the history of semicrossed products, and a thorough discussion of the conjugacy problem, we recommend Davidson, Fuller, and Kakariadis' treatment in \cite{davidson_semicrossed_2018}. For a given action of $P$ on $A$, there are multiple associated semicrossed products $A\times^\calF P$, depending on what family of admissible representations $\calF$ of $P$ one considers. Generally, we have distinct unitary, isometric, and contractive semicrossed products $A\times^\text{un} P$, $A\times^\text{is}P$, and $A\times P$, which satisfy universal properties for ``covariant" contractive/isometric/unitary representations of $P$ with respect to $A$.

Following the programme outlined in \cite{davidson_semicrossed_2017}*{Page 1}, our main question of interest is: If $P$ is a generating subsemigroup of an abelian group $G$, can the C*-envelope of a semicrossed product $A\times^\calF P$ be realized as a full corner of a crossed product $B\rtimes G$ by $G$, for some $G$-C*-algebra $B\supseteq A$? If the action of $P$ on $A$ is by automorphisms, then $A\times^\text{is} P=A\times^\text{un}P$, and the $P$ action extends to $\ast$-automorphisms of the C*-envelope $C_e^\ast(A)$. It follows that
\[
C_e^\ast(A\times^\text{is}P)\cong
C_e^\ast(A)\rtimes G
\]
is a crossed product \cite{davidson_semicrossed_2017}*{Theorem 3.3.1}. If $G=P-P$, and $P$ acts on a C*-algebra $A$ by $\ast$-monomorphisms, then
\begin{equation}\label{eq:envelope_automorphic_extension}
C_e^\ast(A\times^\text{un}P)\cong
\tilde{A}\rtimes G
\end{equation}
is a crossed product for a certain unique minimal C*-algebra $\tilde{A}\supseteq A$ whose $G$-action extends the action of $P$, called the minimal automorphic extension of $A$. Kakariadis and Katsoulis \cite{kakariadis_semicrossed_2012}*{Theorem 2.6} established \eqref{eq:envelope_automorphic_extension} in the case $P=\Z_+$. Laca \cite{laca_endomorphisms_1999} showed how to build the automorphic dilation $\tilde{A}$ in general, and from this Davidson, Fuller, and Kakariadis establish \eqref{eq:envelope_automorphic_extension} in \cite{davidson_semicrossed_2017}*{Theorem 3.2.3}.

Parrott's example \cite{paulsen_completely_2003}*{Chapter 7} of three commuting contractions without a simultaneous isometric dilation, shows that the dilation theory of representations of any semigroup at least as complicated as $\Z_+^3$ is intractable. To make progress, we need to restrict our class of representations $\calF$ if we wish a nice dilation theory for $A\times^\calF P$. Of interest are lattice ordered abelian groups $(G,P)$. These are pairs consisting of a subsemigroup $P$ of a group $G$, where the induced ordering
\[
g\le h\iff 
h-g\in P
\]
makes $G$ a lattice. In the lattice ordered setting, one studies the more tractable class of Nica-covariant representations, first introduced by Nica in \cite{nica_c*-algebras_1992}. Nica-covariance is a $\ast$-commutation type condition which ensures a nice dilation theory. For instance, Li \cites{li_regular_2016,li_regular_2017} showed that every Nica-covariant representation of $P$ has an isometric dilation.

In the Nica-covariant setting, for injective C*-systems \eqref{eq:envelope_automorphic_extension} holds with $A\nctimes P$ in place of $A\times^\text{un}P$. For non-injective systems, it is not possible to embed $A\nctimes P$ into any crossed product $B\rtimes G$ via inclusions $A\subseteq B$ and $P\subseteq G$, because such a system has no faithful unitary covariant pairs. The best one can do is embed $A\nctimes P$ into a \emph{full corner} of a crossed product. For a lattice ordered abelian group $(G,P)$ and an action of $P$ on a C*-algebra $A$, one expects to prove
\begin{equation}\label{eq:envelope_full_corner}
C_e^\ast(A\nctimes P) \cong
p_A(B\rtimes G)p_A,
\end{equation}
is a full corner of a crossed product of some $G$-C*-algebra $B$. Here $A$ embeds into $B$ non-unitally, and $p_A:=1_A$ is the projection coming from the unit in $A$. In the case $(G,P)=(\Z^n,\Z^n_+)$, the result \eqref{eq:envelope_full_corner} was established in the case $n=1$ by Kakariadis and Katsoulis \cites{kakariadis_semicrossed_2011,kakariadis_semicrossed_2012}, and extended to general $n\ge 1$ by Davidson, Fuller, and Kakariadis \cite{davidson_semicrossed_2017}*{Theorem 4.3.7}. Their construction of the $G$-C*-algebra $B$ was in two stages. First, one builds a bigger C*-algebra $B_0\supseteq A$ which has an injective $P$-action dilating the $P$-action on $A$. This is accomplished by a tail-adding technique. Then one takes the minimal automorphic dilation $B:=\tilde{B_0}$.

\subsection{Main Results}

We establish that \eqref{eq:envelope_full_corner} holds for any discrete lattice ordered abelian group $(G,P)$, when $A$ is a C*-algebra (Corollary \ref{cor:envelope_full_corner}). Our approach differs from Davidson, Fuller, and Kakariadis' construction for $P=\Z_+^n$. First, we define a notion of a \emph{Nica-covariant automorphic dilation} of $A$, which is a certain $G$-C*-algebra $B$ with a non-unital embedding $A\subseteq B$. This definition is meant to capture a sufficient set of conditions to get a completely isometric embedding
\[
A\nctimes P\subseteq p_A(B\rtimes G)p_A,
\]
with $p_A:=1_A$. When the dilation $B$ is \emph{minimal}, this is a C*-cover. Then, we show that the Shilov ideal in such a cover has the form $p(I\rtimes G)p$, for a unique maximal $G$-invariant ideal $I\triangleleft B$ not intersecting $A$. Upon taking a quotient by the Shilov ideal,
\[
C_e^\ast(A\nctimes P)\cong
(p_A+I)\left(\frac{B}{I}\rtimes G\right)(p_A+I)
\]
is a full corner of a crossed product. Then it suffices to show that any C*-algebra $A$ with $P$-action has at least one minimal Nica-covariant automorphic dilation. We build one via a direct product construction (Proposition \ref{prop:product_dilation_nc}).

A semicrossed product is a special instance of the tensor algebra of a C*-correspondence \cites{davidson_c*-envelopes_2008,fowler_discrete_1998,kakariadis_contributions_2012,muhly_tensor_1998} (when $P=\Z_+$) or a product system \cites{dinh_discrete_1991,fowler_compactly-aligned_1998,fowler_discrete_2002,fowler_representations_2003,sims_c*-algebras_2007}. Katsoulis and Kribs \cite{katsoulis_tensor_2006} showed that the C*-envelope of the tensor algebra of a C*-correspondence $X$ is the associated Cuntz-Pimsner algebra $\calO_X$, a generalization of the usual crossed product. In \cite{dor-on_tensor_2018}, Dor-On and Katsoulis extend this result and show that the C*-envelope of the Nica tensor algebra $\mathcal{NT}_X^+$ associated to a product system $X$ over $P$ coincides with the associated Cuntz-Nica-Pimsner algebra $\mathcal{NO}_X$ considered by Carlsen, Larsen, Sims, and Vittadello \cite{carlsen_co-universal_2011}, and also coincides with an associated covariance algebra $A\times_X P$ defined by Sehnem \cite{sehnem_c*-algebras_2019}. Our result shows further that, when this product system arises from a C*-dynamical system, this same C*-envelope has the structure of a corner of a crossed product, and so is Morita equivalent to a crossed product.

Before proceeding, we should also direct the reader to the extensive literature on C*-algebras associated to semigroups and semigroup dynamical systems, including \cites{huef_nuclearity_2019,laca_endomorphisms_1999,laca_semigroup_1996,li_semigroup_2012,murphy_crossed_1994,nica_c*-algebras_1992,zahmatkesh_partial-isometric_2017,zahmatkesh_nica-toeplitz_2019}. Following Nica \cite{nica_c*-algebras_1992}, Laca and Raeburn \cite{laca_semigroup_1996} demonstrated for quasi-lattice ordered $(G,P)$ that the universal C*-algebra $C^\ast(G,P)$ for Nica-covariant representations of $P$ has the structure of a semigroup crossed product. Interestingly, we will see (Remark \ref{rem:laca_raeburn}) that our direct product construction of an automorphic dilation reduces to Laca-Raeburn's in the case where $P$ acts on $\C$ trivially.

\subsection{Structure of This Paper} Throughout this section, $(G,P)$ is a (discrete) lattice ordered abelian group, and $P$ acts on a C*-algebra $A$ by $\ast$-endomorphisms. In Section \ref{sec:background}, we review the construction of the semicrossed product, and necessary background on ordered groups and C*-envelopes. Section \ref{sec:main_results} contains our main results. We define the notion of a \emph{minimal Nica-covariant automorphic dilation}, construct such a canonical dilation which we call the \emph{product dilation}, and show that any such dilation always yields a C*-cover of the Nica-covariant semicrossed product $A\nctimes P$ via full corner of a crossed product  (Proposition \ref{prop:automorphic_dilation_embedding}). We show the Shilov ideal arises from a unique maximal $G$-invariant $A$-boundary ideal in any such C*-cover in Theorem \ref{thm:envelope}, and hence show that the C*-envelope of $A\nctimes P$ is a full corner of a crossed product (Corollary \ref{cor:envelope_full_corner}). In two immediate applications, we show that Theorem \ref{thm:envelope} reduces to the known result \eqref{eq:envelope_automorphic_extension} for $A\nctimes P$ in the injective case (Proposition \ref{prop:envelope_injective}), and we compute the unique maximal boundary ideal in the product dilation in the case $P=\Z_+$ (Proposition \ref{prop:envelope_Z}).

Section \ref{sec:Shilov_ideal} is devoted to explicitly computing the Shilov ideal in the C*-cover arising from the product dilation for any Nica-covariant semicrossed product $A\nctimes P$. We do so by describing a unique maximal $G$-invariant boundary ideal $I$ in the product dilation $B$. Then
\[
C_e^\ast(A\nctimes P)\cong
p_A\left(\frac{B}{I}\rtimes G\right)p_A
\]
is a full corner by $p_A:=1_A+I$. Using the explicit construction of $I$ from Section \ref{sec:Shilov_ideal}, in Section \ref{sec:Zn} we show that the $G$-C*-algebra $B/I$ in the case $P=\Z_+^n$ is equivariantly $\ast$-isomorphic to the construction given by Davidson, Fuller, and Kakariadis in \cite{davidson_semicrossed_2017}*{Section 4.3}. So, our description of the C*-envelope reduces to the known result when $P=\Z_+^n$. In Section \ref{sec:applications_examples}, we give some applications both of Theorem \ref{thm:envelope} and the explicit description of $I$ from Section \ref{sec:Shilov_ideal}. In Section 6.1, we establish a simplicity result for the C*-envelope in the commutative case analogous to \cite{davidson_semicrossed_2017}*{Corollary 4.4.9}. In Section 6.2, we show that for totally ordered groups $(G,P)$ which are direct limits of ordered subgroups $(G,P)=\bigcup_\lambda (G_\lambda,P_\lambda)$, such as $\Q=\bigcup_n \Z/n!$, we have
\[
C_e^\ast(A\nctimes P)=
\varinjlim_{\lambda} C_e^\ast(A\nctimes P_\lambda)
\]
naturally, as long as $P$ acts on $A$ by surjections. This result is sharp and fails for non-totally ordered groups and non-surjective actions.

\subsection*{Acknowledgements} The author would like to thank Kenneth Davidson and Matthew Kennedy for their helpful comments and useful discussions. The author is also grateful to the referee and to Evgenios Kakariadis for some helpful suggestions that improved the paper.

\section{Background} \label{sec:background}

In this paper, a (discrete, unital) \textbf{semigroup} $P$ is a set equipped with an associative binary operation, and we require that $P$ contains a two-sided identity element. We are primarily interested in abelian semigroups. In the abelian setting, we will always denote the semigroup operation by $+$ and the identity element by $0$. A semigroup homomorphism is a function between semigroups preserving the semigroup operations and the identity.

If $A$ is a C*-algebra, an \textbf{ideal} $I\triangleleft A$ always means a closed, two-sided ideal. We make frequent use of the following two inductivity properties of ideals in C*-algebras. Firstly, if \[
A=
\overline{\bigcup_{\lambda \in \Lambda}A_\lambda} \cong
\varinjlim_{\lambda \in \Lambda}A_\lambda
\]
is an internal direct limit of C*-subalgebras $A_\lambda$, and $I\triangleleft A$ is an ideal, then
\[
I=
\overline{\bigcup_{\lambda \in \Lambda}I\cap A_\lambda}.
\]
In particular, $I=\{0\}$ if and only if $I\cap A_\lambda=\{0\}$ for all $\lambda\in \Lambda$. Secondly, if $\{I_\lambda \mid \lambda \in \Lambda\}$ is a family of ideals in $A$ that is directed under inclusion, then $I:=\overline{\bigcup_{\lambda \in \Lambda}I_\lambda}$ is also an ideal in $A$.

Let $P$ be a semigroup. An \textbf{(operator algebra) dynamical system} $(A,\alpha,P)$ over $P$ consists of an operator algebra $A$ and a semigroup action $\alpha$ of $P$ on $A$ by completely contractive algebra endomorphisms. That is, there is a distinguished (unital) semigroup homomorphism
\[
p\mapsto \alpha_p:P\to \End(A).
\]
We do not require the $\alpha_p$ to be automorphisms. We will say that $(A,\alpha,P)$ is \textbf{injective/surjective/automorphic} if each $\alpha_p$ is injective/surjective/automorphic. When $A$ has an identity $1_A$ and each $\alpha_p$ is unital, we call $(A,\alpha,P)$ a \textbf{unital} dynamical system. If $A$ is a C*-algebra, and hence each $\alpha_p$ is an $\ast$-endomorphism, then $(A,\alpha,P)$ is a \textbf{C*-dynamical system}.

Let $G$ be an abelian group. A subsemigroup $P\subseteq G$ is a \textbf{positive cone} if $P\cap (-P)=\{0\}$, and a \textbf{spanning cone} if in addition $G=P-P$. Any positive cone $P\subseteq G$ induces a partial order on $G$ by defining
\[
g\le h \iff
h-g \in P.
\]
This ordering respects the group operation $+$. A \textbf{lattice ordered abelian group} $(G,P)$ consists of an abelian group $G$ and a spanning cone $P\subseteq G$ such that the partial order $\le$ induced by $P$ on $G$ makes $G$ into a lattice. That is, for any $g,h\in G$, the $\{g,h\}$ has a least upper bound $g\vee h$ and a greatest lower bound $g\wedge h$. If $(G,P)$ is a lattice ordered abelian group, we also refer to $P$ as a \textbf{lattice ordered abelian semigroup}.

\begin{eg}
The pair $(\Z^n,\Z^n_+)$ forms a lattice ordered abelian group. Here, a  dynamical system $(A,\alpha,\Z^n_+)$ consists of a choice of $n$ commuting completely contractive endomorphisms of $A$, which we usually just write as $\alpha_1,\ldots,\alpha_n\in \End(A)$.
\end{eg}

\begin{eg}
Any \textbf{totally ordered} group $(G,P)$ is automatically lattice ordered. For instance, $(\Q,\Q_+)$ and $(\R,\R_+)$ are both totally ordered groups. If $P\subseteq \Z^n$ is the set of elements larger than $(0,\ldots,0)$ in the lexicographic ordering of $\Z^n$, then $(\Z^n,P)$ is totally ordered, and the induced ordering is lexicographic.
\end{eg}

A representation $T:P\to B(H)$ is a (unital) semigroup homomorphism, and we usually write $T(p)=T_p$. The representation $T$ is \textbf{contractive/isometric/unitary} whenever each $T_p$ is contractive/isometric/unitary. If $(G,P)$ is a lattice ordered group, a contractive representation $T:P\to B(H)$ is \textbf{Nica-covariant} if whenever $p,q\in P$ satisfy $p\wedge q=0$, we have $T_p T_q^\ast=T_q^\ast T_p$, so $T_p$ and $T_q$ not only commute, but $\ast$-commute \cite{nica_c*-algebras_1992}. If $V:P\to B(H)$ is an isometric representation, $V$ is Nica-covariant if and only if
\[
V_pV_p^\ast \, V_qV_q^\ast =
V_{p\vee q}V_{p\vee q}^\ast.
\]
That is, the range projections of the $V_p$'s give a lattice homomorphism $P\to \text{proj}(H)$. A representation $T$ of $\Z_+^n$ is Nica-covariant if and only if the generators $T_1,\ldots,T_n$ $\ast$-commute, and in this case we can find a simultaneous dilation to isometries $V_1,\ldots,V_n$, which yield an isometric Nica-covariant representation $V$ that dilates $T$ \cite{davidson_semicrossed_2017}*{Theorem 2.5.10}. More generally, for any lattice ordered abelian semigroup $P$, any contractive Nica-covariant representation $T:P\to B(H)$ has an isometric Nica-covariant co-extension. For any lattice ordered abelian semigroup $P$, Li \cite{li_regular_2016} showed that any Nica-covariant representation $T:P\to B(H)$ extends to a completely positive definite function on the whole group, and so $T$ co-extends to an isometric Nica-covariant representation of $P$ by \cite{davidson_semicrossed_2017}*{Theorem 2.5.10}

Let $(A,\alpha,P)$ be a dynamical system over an abelian semigroup $P$. A \textbf{covariant pair} 
\[
(\pi,T):(A,P)\to B(H)
\]
consists of a $\ast$-homomorphism $\pi:A\to B(H)$, and a representation $T:P\to B(H)$, such that, if $a\in A$ and $p\in P$,
\begin{equation}\label{eq:covariant_pair}
\pi(a)T_p=
T_p\pi(\alpha_p(a)).
\end{equation}
We say $(\pi,T)$ is \textbf{unitary/isometric/contractive/Nica-covariant} when $T$ is so. Let $\calF$ be a sufficiently well behaved family of representations of $P$ on Hilbert space (cf. \cite{davidson_semicrossed_2017}*{Definition 2.1.1}). For our purposes, $\calF$ is always one of ``$\text{un}$" (unitary representations), ``$\text{is}$" (isometric representations), ``$\text{c}$" (contractive representations), or when $P$ is a lattice ordered abelian semigroup, ``$\text{nc}$" (Nica-covariant representations). 

The \textbf{semicrossed product}
\[
A\times_\alpha^\calF P
\]
is an operator algebra defined by the following universal property \cite{davidson_semicrossed_2017}*{Section 3.1}. There is a covariant pair $(i,v):(A,P)\to A\times_\alpha^\calF P$ such that whenever $(\pi,T):(A,P)\to B(H)$ is a covariant pair with $T\in \calF$, there is a unique completely contractive homomorphism
\begin{equation}\label{eq:semicrossed_covariance}
\pi\times T:A\times_\alpha^\calF P\to B(H)
\end{equation}
with $(\pi\times T)\circ i = \pi$ and $(\pi\times T)\circ v=T$. Concretely, $A\times_\alpha^\calF P$ is densely spanned by formal monomials $v_p a$, for $p\in P$ and $a\in A$, which satisfy the relation
\[
(v_p a)\cdot (v_qb) =
v_{p+q}\alpha_q(a)b.
\]
Indeed, one can construct $A\times^\calF_\alpha P$ by starting with the algebraic tensor product
\[
\C[P]\odot A,
\]
defining a multiplication relation
\[
(\delta_p\otimes a)\cdot(\delta_q\otimes b)=
\delta_{p+q}\otimes (\alpha_q(a)b),
\]
and completing in the universal operator algebra norm defined by
\[
\|X\|:=
\sup\left\{\left\|(\pi\times T)^{(n)}(X)\right\|\;\middle|\; (\pi,T) \text{ is a covariant pair with }T\in \calF\right\},
\]
for any $X\in M_n(\C[P]\odot A)$. Here $\pi\times T:\delta_p\otimes a\mapsto T_p\pi(a)$ defines a homomorphism on $\C[P]\odot A$. When the action $\alpha$ is clear, we usually just write $A\times^\calF P$. We do not omit $\calF$, because it is standard that
\[
A\times_\alpha P:=
A\times_\alpha^\text{c} P
\]
always denotes the contractive semicrossed product. Note that commutativity of the operation in $P$ is used to prove the multiplication on $\C[P]\odot A$ is associative.

We are primarily interested in the Nica-covariant semicrossed product $A\nctimes P$, when $(G,P)$ is a lattice ordered abelian group. By \cite{davidson_semicrossed_2017}*{Proposition 4.2.1} and \cite{li_regular_2016}, $A\nctimes P$ is also universal for \emph{isometric} Nica-covariant pairs, and these completely norm the algebra $A\nctimes P$. In fact, there is a distinguished isometric Nica-covariant pair for any pair $(A,P)$. Let $\pi:A\to B(H)$ be any completely isometric representation. Define a pair $(\tilde{\pi},V):(A,P)\to B(H\otimes \ell^2(P))$ by
\[
\tilde{\pi}(a)(x\otimes \delta_p) =
\alpha_p(a)x\otimes \delta_p,\quad
V_q(x\otimes \delta_p) =
x\otimes \delta_{p+q}.
\]
Then $(\tilde{\pi},V)$ is an isometric Nica-covariant pair, and we call $\tilde{\pi}\times V:A\nctimes P\to B(H)$ the \textbf{Fock representation} (induced by $\pi$) for $A\nctimes P$. By \cite{davidson_semicrossed_2017}*{Theorem 4.2.9}, any Fock representation is completely isometric. This is a key tool which makes it easy to prove that $A\nctimes P$ embeds completely isometrically into a crossed product.

Let $G$ be an abelian group and $(B,\beta,G)$ a C*-dynamical system over $G$. In this paper, we use the nonstandard convention that the crossed product $B\rtimes_\beta G$ is the universal C*-algebra generated by monomials
\[
u_ga,\quad 
g\in G, a\in B
\]
satisfying $u_ga\cdot u_hb=u_{g+h}\beta_h(a)b$, or when $B$ is unital,
\[
u_g^\ast au_g =
\beta_g(a).
\]
Usually one takes the convention that $u_g a u_g^\ast = \beta_g(a)$. This backwards convention is only valid because $G$ is abelian, so $g\mapsto u_g^\ast$ defines a unitary representation of $G$. Clearly this construction is isomorphic to the usual crossed product, so we lose no generality. What we gain is an alignment with the semicrossed covariance relations \eqref{eq:covariant_pair} and \eqref{eq:semicrossed_covariance}. Indeed
\[
B\rtimes_\beta G\cong
B\times^\text{un}_\beta G\cong
B\times^\text{is}_\beta G
\]
is also a semicrossed product, and a C*-algebra with the obvious $\ast$-structure. As with semicrossed products, we usually write $B\rtimes G$ when the action $\beta$ is clear.

Generally, for any dynamical system $(A,\alpha,P)$, the semicrossed product $A\times_\alpha^\calF P$ is a (non-selfadjoint) operator algebra, even when $A$ is a C*-algebra. Let $A$ be any operator algebra. A \textbf{C*-cover} $\varphi:A\to B$ for $A$ consists of a C*-algebra $B$, and a unital completely isometric homomorphism $\varphi$ such that $B=C^\ast(\varphi(A))$. The \textbf{C*-envelope} $C_e^\ast(A)$ is a co-universal or terminal C*-cover $\iota:A\to C_e^\ast(A)$. That is, whenever $\varphi:A\to B$ is a C*-cover, there is a $\ast$-homomorphism $\pi:B\to C_e^\ast(A)$ such that
\[\begin{tikzcd}
    A \arrow[r,"\varphi"] \arrow[dr,swap,"\iota"]& B \arrow[d,"\pi"] \\
    & C_e^\ast(A)
\end{tikzcd}\]
commutes. The homomorphism $\pi$ is necessarily unique and surjective. The C*-envelope exists and is unique up to a $\ast$-homomorphism fixing $A$ \cite{hamana_injective_1979}. In fact, it can be produced from any C*-cover. If $\varphi:A\to B$ is a C*-cover, a \textbf{boundary ideal} $I\triangleleft B$ is an ideal such that the quotient map $q:B\to B/I$ is completely isometric on $\varphi(A)$. Existence of the C*-envelope implies that there is a unique maximal boundary ideal in $B$ for $A$ called the \textbf{Shilov ideal}, and $q\varphi:A\to B\to B/I$ is a C*-envelope for $A$ \cite{arveson_subalgebras_1969}.

\section{Main Results} \label{sec:main_results}

Let $(G,P)$ be a lattice ordered abelian group, and let $(A,\alpha,P)$ be a unital C*-dynamical system over $P$. Our goal is to embed the Nica-covariant semicrossed product $A\nctimes_\alpha P$ into a crossed product $B\rtimes_\beta G$. Here, $A$ should be a C*-subalgebra of $B$ and the action $\beta$ of $G$ on $A$ should extend or dilate the action $\alpha$ of $P$. Write
\[
B\rtimes_\beta G =\cspn\{u_g b\mid g\in G,\, b\in B\}.
\]
We might hope to embed $A\nctimes P$ in $B\rtimes G$ via a map of the form $\iota\times u$, where $\iota:A\to B$ is some unital $\ast$-monomorphism that intertwines $\alpha$ and $\beta$. However, this is impossible whenever any $\alpha_p$ has kernel. Indeed, if $a\in \ker \alpha_p\subseteq A$ is nonzero, then we would require $\iota(a)\ne 0$, but
\[
0 =
\iota(\alpha_p(a)) =
u_p^\ast \iota(a)u_p.
\]
This is impossible when $u_p$ is unitary.

In the non-injective case, the best we can do is embed $A\nctimes P$ into a corner of $B\rtimes G$. We do this by taking a nonunital embedding $\iota:A\to B$. Then, $p_A:=\iota(1_A)$ is a projection in $B$. Consequently,
\[
up_A:p\mapsto u_pp_A
\]
defines an isometric representation of $P$ in the corner $p_A(B\rtimes G)p_A$. The following definition is meant to capture a set of sufficient conditions for $(\iota,up_A)$ to be a Nica-covariant covariant pair, and give an embedding $\iota\times up_A:A\nctimes P\to p_A(B\rtimes G)p_A$ (Proposition \ref{prop:automorphic_dilation_embedding}).

\begin{defn}\label{def:automorphic_dilation}
Let $(G,P)$ be a lattice ordered abelian group. Suppose $(A,\alpha,P)$ is a C*-dynamical system over $P$. An \textbf{automorphic dilation} $(B,\beta,G)$ is a C*-dynamical system $(B,\beta,G)$ together with
\begin{itemize}
\item [(1)] a $\ast$-monomorphism $\iota:A\to B$, such that
\item [(2)] for all $a,b\in A$ and $p\in P$,
\[
\iota(a)\beta_p(\iota(b))=
\iota(a\alpha_p(b)).
\]
\end{itemize}
By taking adjoints we have that $\beta_p(\iota(A))\iota(b)=\iota(\alpha_p(a)b)$. Moreover we say that $(B,\beta,G)$ is a \textbf{Nica-covariant} automorphic dilation if in addition
\begin{itemize}
\item [(3)] for all $a,b\in A$ and $g,h\in G$, we have
\[
\beta_g(\iota(a))\,\beta_h(\iota(b)) =
\beta_{g\wedge h}(\iota(\alpha_{g-g \wedge h}(a)\alpha_{h-g\wedge h}(b))).
\]
\end{itemize}
We call an automorphic dilation $(B,\beta,G)$ \textbf{minimal} if
\begin{itemize}
\item [(4)] $\iota(A)$ generates $B$ as a $G$-C*-algebra, i.e.
\[
B=
C^\ast\left(\bigcup_{g\in G}\beta_g\iota(A)\right).
\]
\end{itemize}
\end{defn}

We are primarily concerned with minimal Nica-covariant automorphic dilations, which satisfy all of (1)-(4). Note that if the automorphic dilation $(B,\beta,G)$ is both minimal and Nica-covariant then property (3) above implies that
\[
\sum_{g\in G} \beta_g\iota(A)
\]
is a $\ast$-subalgebra, and hence
\[
B=
\overline{\sum_{g\in G} \beta_g\iota(A)}.
\]
We are also mostly concerned with the unital case.

\begin{rem}\label{rem:unital_automorphic_dilation}
Let $(G,P)$ be a lattice ordered abelian group, and let $(A,\alpha,P)$ be a unital C*-dynamical system. Suppose $(B,\beta,G)$ is an automorphic C*-dynamical system, with a (possibly nonunital) $\ast$-monomorphism $\iota:A\to B$. Setting $p_A:=\iota(1_A)$, it is straightforward to check that properties (2) and (3) in Definition \ref{def:automorphic_dilation} are equivalent to:
\begin{itemize}
\item [(2')] For all $a\in A$ and $p\in P$,
\[
p_A\beta_p(\iota(a))=\iota(\alpha_p(a))\quad (=\beta_p(\iota(a))p_A),
\]
and
\item [(3')] for all $g,h\in G$,
\[
\beta_g(p_A)\beta_h(p_A)=\beta_{g\wedge h}(p_A).
\]
\end{itemize}
Clearly if (2) and (3) hold, so do (2') and (3'). Conversely, if both (2') and (3') hold, then (2) holds because $\iota(a)=\iota(a)p_A$. Given $a,b\in A$ and $g,h\in G$, using both (2') and (3') we find
\begin{align*}
\beta_g(\iota(a))\beta_h\iota(b)) &=
\beta_g(\iota(a))(\beta_g(p_A)\beta_h(p_A))\beta_h\iota(b)) \\ &=
\beta_g(\iota(a))\beta_{g\wedge h}(p_A)\beta_h(\iota(b)) \\ &=
\beta_{g\wedge h}(\beta_{g-g\wedge h}(a)p_A\beta_{h-g\wedge h}(b)) \\ &=
\beta_{g\wedge h}(\iota(\alpha_{g-g\wedge h}(a)\alpha_{h-g\wedge h}(b))),
\end{align*}
showing (3) holds.
\end{rem}

The reason we assign property (3) the name ``Nica-covariant" is because in the unital case, the identity $\beta_g(p_A)\beta_h(p_A)=\beta_{g\wedge h}(p_A)$ ensures that the isometric semigroup representation $p\mapsto u_pp_A\in p_A(B\rtimes G)p_A$ is Nica-covariant. Indeed,
\[
(u_pp_A)(u_pp_A)^\ast =
u_p p_A u_p^\ast =
\beta_{-p}(p_A).
\]
So if (3) holds, the element $(u_pp_A)(u_pp_A)^\ast \cdot (u_qp_A)(u_qp_A)^\ast$ equals
\[
\beta_{-p}(p_A)\beta_{-q}(p_A) =
\beta_{(-p)\wedge (-q)}(p_A) =
\beta_{- p\vee q}(p_A) =
(u_{p\vee q} p_A)(u_{p\vee q}p_A)^\ast.
\]
In a minimal Nica-covariant automorphic dilation, the projections $\beta_p(p_A)$ for $p\in P$ are central and in fact form an approximate identity.

\begin{lem}\label{lem:p_A_approx_identity}
Let $(G,P)$ be a lattice ordered abelian group. Suppose $(A,\alpha,P)$ is a unital C*-dynamical system, with $(B,\beta,G)$ a minimal Nica-covariant automorphic dilation. Considering $P$ (a lattice) as a directed set, the net
\[
(\beta_p(p_A))_{p\in P}
\]
is an increasing approximate identity for $B$, consisting of central projections.
\begin{proof}
Suppose $p\le q$ in $P$. Then
\[
\beta_p(p_A)\beta_q(p_A)=
\beta_{p\wedge q}(p_A) =
\beta_p(p_A).
\]
Therefore $\beta_p(p_A)\le \beta_q(p_A)$, since both are projections. These projections are central, because for an element of the form $\beta_g\iota(a)$, where $a\in A$ and $g\in G$, we have
\[
\beta_p(p_A)\beta_g(\iota(a))=
\beta_{p\wedge g}\iota(\alpha_{g-p\wedge g}(a)) =
\beta_g(\iota(a))\beta_p(p_A).
\]
Here, we have used property (3) in Definition \ref{def:automorphic_dilation} and the fact that each $\alpha_p$ fixes $1_A$. Since $(G,P)$ is lattice ordered, any element $g\in G$ is dominated by an element $p=g\vee 0\in P$. Further, when $p\ge g$, we have $p\wedge g=g$ and the same computation shows
\[
\beta_p(p_A)\beta_g\iota(a) =
\beta_g\iota(a).
\]
Thus $(\beta_p(p_A))_{p\in P}$ is an approximate identity for $\beta_g\iota(A)$, which commutes with $\beta_g\iota(A)$. Since the automorphic dilation $(B,\beta,G)$ is minimal,
\[
B=\overline{\sum_{g\in G}\beta_g\iota(A))}.
\]
Thus each $\beta_p(p_A)$ is central. Since the net $(\beta_p(p_A))_p$ is norm-bounded, a standard $\varepsilon/3$ argument shows it is an approximate identity on all of $B$.
\end{proof}
\end{lem}

The key observation is that any C*-dynamical system over a lattice ordered abelian semigroup admits a minimal Nica-covariant automorphic dilation. In fact, we can build one with an infinite product construction.

\begin{defn}\label{def:product_dilation}
Let $(G,P)$ be a lattice ordered abelian group, and $(A,\alpha,P)$ a C*-dynamical system. We construct a minimal Nica-covariant automorphic dilation as follows. Define the $\ast$-monomorphism
\[
\iota:A\to \prod_G A
\]
by
\[
\iota(a)_g=
\begin{cases}
    \alpha_g(a) & g\in P, \\
    0 & g\not \in P.
\end{cases}
\]
Throughout, $[x]_g$ or simply $x_g$ always denotes the $g$'th element of a tuple $x\in \prod_GA$. Then, $G$ acts on $\prod_G A$ by the ``left-shift" $\beta:G\to \End(\prod_GA)$, where
\[
[\beta_g(x)]_h=
x_{h+g}.
\]
Set
\begin{equation}\label{eq:B_def}
B:=
C^\ast\left(\bigcup_{g\in G} \beta_g\iota(A)\right) =
\overline{\sum_{g\in G} \beta_g\iota(A)}.
\end{equation}
Then, $(B,\beta,G)$ is a minimal Nica-covariant automorphic dilation of $(A,\alpha,P)$, which we call the \textbf{product dilation} of $(A,\alpha,P)$.
\end{defn}

\begin{prop}\label{prop:product_dilation_nc}
The product dilation $(B,\beta,G)$ is a minimal Nica-covariant automorphic dilation of $(A,\alpha,P)$.
\end{prop}

\begin{proof}
Given $a,b\in A$, $p\in P$, and $g\in G$, we compute
\[
[\iota(a)\beta_p\iota(b)]_g =
\begin{cases}
    \alpha_g(a)\alpha_{g+p}(b) & g\ge 0, \\
    0 & \text{else},
\end{cases}
\]
which equals $[\iota(a\alpha_p(b))]_g$. Thus $\iota(a)\beta_p\iota(b)=\iota(a\alpha_p(b))$. So, $(B,\beta,G)$ is an automorphic dilation of $(A,\alpha,P)$. By \eqref{eq:B_def}, this dilation is minimal. Let $a,b\in A$ and $g,h,k\in G$. Then
\[
[\beta_g\iota(a)\beta_h\iota(b)]_k  =
\begin{cases}
    \alpha_{g+k}(a)\alpha_{h+k}(b) & k\ge -g \text{ and } k\ge -h, \\ 
    0 & \text{else.}
\end{cases}
\]
Because $k\ge -g$ and $k\ge -h$ if and only if $k\ge (-g)\vee (-h)=-(g\wedge h)$, it follows that
\[
\beta_g\iota(a)\beta_h\iota(b)=
\beta_{g\wedge h}\iota(\alpha_{g-g\wedge h}(a)\alpha_{h-g\wedge h}(b)),
\]
so the dilation is Nica-covariant.
\end{proof}

\begin{rem}\label{rem:zahmatkesh}
While finalizing this paper, the author was made aware that the product dilation defined here was defined first by Zahmatkesh for totally ordered abelian groups in \cite {zahmatkesh_partial-isometric_2017}, and for general lattice ordered abelian groups in \cite{zahmatkesh_nica-toeplitz_2019}. In Proposition \ref{prop:automorphic_dilation_embedding}, we prove that a full corner of the crossed product associated to the product dilation is a C*-cover of the semicrossed product $A\nctimes P$. Zahmatkesh proves in \cite{zahmatkesh_nica-toeplitz_2019} that this same full corner is the universal C*-algebra associated to Nica-Toeplitz covariant representations of $(A,\alpha,P)$.
\end{rem}

\begin{eg}\label{eg:Z_product_dilation}
It is most instructive to consider the product dilation of a unital system in the case $(G,P)=(\Z,\Z_+)$. Here, we embed $A$ in $\prod_\Z A$ via
\[
\iota(a) :=
(\ldots,0,0,a,\alpha(a),\alpha^2(a),\ldots),
\]
the ``$a$" occurring at index $0$. Then, $\Z$ acts on $\prod_\Z A$ by the backwards bilateral shift $\beta$. This is an automorphic dilation, because $p_A=(\ldots,0,0,1,1,1,\ldots)$ and 
\[
p_A\beta\iota(a) =
(\ldots,0,0,\alpha(a),\alpha^2(a),\ldots) =
\iota\alpha(a).
\]
\end{eg}

\begin{rem}\label{rem:laca_raeburn}
When $A=\C$ and $P$ acts trivially, the product dilation $(B,\beta,G)$ is the C*-algebra $B_P$ that Laca and Raeburn define in \cite{laca_semigroup_1996}*{Section 2}. In terms of their notation, $1_0=p_A$, and for $p\in P$, $1_p=\beta_{-p}(p_A)$. Nica-covariance of the dilation $B_P$ is seen in Equation (1.2) in \cite{laca_semigroup_1996}.
\end{rem}

As promised, the Nica-covariant semicrossed product $A\nctimes P$ embeds into the crossed product of any Nica-covariant automorphic dilation.

\begin{prop}\label{prop:automorphic_dilation_embedding}
Let $(G,P)$ be a lattice ordered abelian group. Let $(A,\alpha,P)$ be a unital C*-dynamical system. Suppose $(B,\beta,G)$ is a Nica-covariant automorphic dilation of $(A,\alpha,P)$, with $\ast$-embedding $\iota:A\to B$. With $p_A=\iota(1_A)$, there is a completely isometric homomorphism
\[
\varphi=\iota\times up_A:A\nctimes_\alpha P\to B\rtimes_\beta G,
\]
where $(up_A)_p=u_pp_A$. Moreover, if $(B,\beta,P)$ is a minimal automorphic dilation, then
\[
C^\ast(\varphi(A\nctimes_\alpha P))=
p_A(B\rtimes_\beta G)p_A
\]
is a full corner of $B\rtimes_\beta G$.
\end{prop}

\begin{proof}
As shown after Remark \ref{rem:unital_automorphic_dilation}, Nica-covariance of the dilation $(B,\beta,G)$ ensures that $up_A:P\to p_A(B\rtimes G)p_A$ is an isometric Nica-covariant representation of $P$. Further, because $p_A=\iota(1_A)$, $\iota$ maps $A$ into $p_A(B\rtimes G)p_A$. The pair $(\iota,up_A)$ is covariant, as for $a\in A$ and $p\in P$,
\[
\iota(a)u_pp_A =
u_p\beta_p\iota(a)p_A =
u_p\iota(\alpha_p(a)) =
u_pp_A\iota(\alpha_p(a)).
\]
By the universal property, there exists a completely contractive homomorphism
\[
\varphi=
\iota\times up_A:A\nctimes P\to p_A(B\rtimes G)p_A \subseteq B\rtimes G.
\]

We have to show that $\varphi$ is completely isometric. Fix any faithful nondegenerate representation $\pi:B\to B(H)$. As $G$ is abelian, the left regular representation
\[
U\times \tilde{\pi}:B\rtimes G\to B(H\otimes \ell^2(G))
\]
is faithful. Then $H\otimes \ell^2(P)\subseteq H\otimes \ell^2(G)$ is a $\tilde{\pi}(B)$ and $U(P)$-invariant subspace. Let 
\begin{align*}
\sigma:=\tilde{\pi}\circ \iota\vert_{H\otimes \ell^2(P)}:A &\to B(H\otimes \ell^2(P)), \quad \text{and}\\
V:= U\vert_{H\otimes \ell^2(P)}:P&\to B(H\otimes \ell^2(P)).
\end{align*}
Then, it is immediate that $(\sigma,V)$ is a Nica-covariant covariant pair for $(A,P)$, and by definition, $\sigma\times V$ is the Fock representation of $A\nctimes P$ on $H\otimes \ell^2(P)$. By \cite{davidson_semicrossed_2017}*{Theorem 4.2.9}, the Fock representation is completely isometric. Let $\kappa:B(H\otimes \ell^2(G))\to B(H\otimes \ell^2(P))$ be the compression map. The diagram
\[\begin{tikzcd}
    A\nctimes P \arrow[r,"\varphi"] \arrow[d,"\sigma\times V"]& B\rtimes G \arrow[d,"\pi\times U"] \\
    B(H\otimes \ell^2(P)) & B(H\otimes \ell^2(G)) \arrow[l,swap,"\kappa"]
\end{tikzcd}\]
commutes. As the vertical maps are complete isometries, and $\kappa$ is a complete contraction, it follows that $\varphi$ is completely isometric, as claimed.

Now suppose that $(B,\beta,G)$ is minimal. We claim that the corner $p_A(B\rtimes G)p_A$ is full and generated by $\varphi(A\nctimes P)$. This is a full corner, because $p_A(B\rtimes G)p_A$ contains $A\subseteq p_ABp_A$, and as $A$ generates $B$ as a $G$-C*-algebra, the ideal that $A$ generates in $B\rtimes G$ is everything. By minimality,
\[
B=
\overline{\sum_{g\in G}\beta_g\iota(A)},
\]
so $B$ is densely spanned by monomials $x=u_g\beta_h\iota(a))$ for $a\in A$ and $g,h\in G$. Given such a monomial, as $p_A$ is central in $B$,
\begin{align*}
p_Axp_A &= p_A u_gp_A\beta_h\iota(a)p_A \\ &=
p_Au_gp_Au_h^\ast \iota(a)u_hp_A \\ &=
(u_{g_-} p_A)^\ast(u_{g_+} p_A)(u_hp_A)^\ast \iota(a)(u_hp_A).
\end{align*}
Here, since $P$ is a spanning cone we have written $g=g_+-g_-$, where $g_\pm\in P$. Thus, $x\in C^\ast(\iota,up_A)$ and $p_A(B\rtimes G)p_A\subseteq C^\ast(\iota,up_A)$. Conversely, since $(\iota,up_A)$ is a Nica-covariant isometric pair, by \cite{davidson_semicrossed_2017}*{Proposition 4.2.3}, $C^\ast(\iota,up_A)$ is densely spanned by monomials $y=(u_pp_A)\iota(a)(u_q p_A)^\ast$, for $a\in A$ and $p,q\in P$. Given $p\in P$, we have 
\[
p_Au_p p_A=u_p\beta_p(p_A)p_A=u_p\beta_{p\wedge 0}(p_A) =u_pp_A,
\]
and by taking adjoints $p_Au_p^\ast =p_Au_p^\ast p_A$. Then, for such a monomial $y$, we find
\begin{align*}
y&=
u_p p_A\iota(a)p_Au_q^\ast \\ &=
p_A u_p p_A\iota(a)p_Au_q^\ast p_A =
p_Ayp_A.
\end{align*}
This proves $C^\ast(\iota,up_A)=p_A(B\rtimes G)p_A$, as desired.
\end{proof}

Proposition \ref{prop:automorphic_dilation_embedding} asserts that $p_A(B\rtimes G)p_A$ is a C*-cover of $A\nctimes P$. To find the C*-envelope $C_e^\ast(A\nctimes P)$, it suffices to describe the Shilov ideal. In Theorem \ref{thm:envelope} we will show that the Shilov ideal arises as a corner of a crossed product $I\rtimes G$, where $I\triangleleft B$ is some $G$-invariant ideal in $B$.  

\begin{defn}\label{def:boundary_ideal}
Let $A$ be an operator algebra, $B$ a C*-algebra, and suppose there is a completely isometric homomorphism $\iota:A\to B$. A (closed) ideal $I\triangleleft B$ is called an \textbf{$A$-boundary ideal} (with respect to $\iota$) if the quotient map $B\to B/I$ restricts to be completely isometric on $A$.
\end{defn}

Note that if $A$ is a C*-algebra in Definition \ref{def:boundary_ideal}, then $I$ is a boundary ideal if and only if the quotient map $B\to B/I$ is faithful on $I$. This occurs if and only if $A\cap I=\{0\}$.

\begin{rem}\label{rem:quotient_dilation}
It is routine to check that if $(B,\beta,G)$ is a C*-dynamical system, and $I\triangleleft B$ is a $\beta$-invariant ideal, then $(B/I,\tilde{\beta},G)$ is also a C*-dynamical system. Here $\tilde{\beta}_g(b+I):=\beta_g(b)+I$ is well defined, by invariance of $I$. The quotient map $q:B\to B/I$ is $G$-equivariant.

Further, suppose that $(G,P)$ is a lattice ordered abelian group, and $(B,\beta,G)$ is an automorphic dilation of $(A,\alpha,P)$ with inclusion $\iota:A\to B$. If $I\triangleleft B$ is a $\beta$-invariant $A$-boundary ideal (meaning $\iota(A)\cap I=\{0\}$), then $(B/I,\tilde{\beta},G)$ is also an automorphic dilation of $(A,\alpha,P)$, because $q\iota$ is faithful on $A$. Moreover, if $(B,\beta,G)$ is Nica-covariant or minimal, then so too is $(B/I,\tilde{\beta},G)$, which easily follows from equivariance of $q$.
\end{rem}

The following lemma summarizes that under reasonable hypotheses we can ``commute" taking quotients with either taking corners or crossed products.

\begin{lem}\label{lem:ideal_facts}
\begin{inparaenum}
\item [(i)] Suppose $C$ is a C*-algebra, and $p\in C$ is a projection. Let $J\triangleleft pCp$ be an ideal. If $K=\langle J\rangle_C=\overline{CJC}$ is the ideal $J$ generates in $C$, then $J=pKp$. Moreover, there is a canonical isomorphism
\[
\frac{pCp}{J}\cong
(p+K)\left(\frac{C}{K}\right)(p+K).
\]\medskip

\item [(ii)] Suppose $(B,\beta,G)$ is an automorphic C*-dynamical system over an abelian group $G$. Let $I\triangleleft B$ be a $G$-invariant ideal. Then the natural map $B\rtimes G\to (B/I)\rtimes G$ induces an isomorphism
\[
\frac{B\rtimes_\beta G}{I\rtimes_\beta G}\cong \frac{B}{I}\rtimes_{\tilde{\beta}} G.
\]
\end{inparaenum}
\end{lem}

\begin{proof}
(i) Since $J\subseteq K$, certainly $J=pJp\subseteq pKp$. Conversely, for any term $ajb$, for $a,b\in C$ and $j\in J\subseteq pAp$, the product
\[
p(ajb)p=
pa(pjp)bp=
(pap)j(pbp)
\]
lies in $J$, since $J\triangleleft pCp$. Thus $pKp=J$. Restricting the quotient map $C\to C/K$ gives a $\ast$-homomorphism with range $(p+K)(C/K)(p+K)$ and kernel $K\cap pCp=pKp=J$, so the stated isomorphism follows.
\bigskip

(ii) This follows because $G$ is abelian, and hence an exact group \cite{brown_c*-algebras_2008}*{Theorem 5.1.10}.
\end{proof}

Recall that when $G$ is an abelian group, the compact dual group $\widehat{G}$ has a natural gauge action $\gamma$ on any crossed product $B\rtimes G$, which satisfies
\[
\gamma_\chi(u_gb)=
\chi(g)u_gb.
\]
Consequently, there is a faithful expectation
\[
E_{\widehat{G}}:B\rtimes G \to B\rtimes G
\]
with range $B$, defined by the formula
\[
E_{\widehat{G}}(x) =
\int_{\widehat{G}}\gamma_\chi(x)\; d\chi.
\]
Here $d\chi$ denotes integration against Haar measure.

\begin{lem}\label{lem:invariant_ideal}
Suppose $(B,\beta,G)$ is an automorphic C*-dynamical system over an abelian group $G$. Let $J\triangleleft B\rtimes_\beta G$ be an ideal. Then $J$ is invariant under the gauge action of $\widehat{G}$ if and only if $J=I\rtimes_\beta G$, where $I=J\cap B\triangleleft B$ is a $\beta$-invariant ideal in $B$.
\begin{proof}
Since $\widehat{G}$ acts diagonally on the spanning monomials $u_gb$ in $B\rtimes G$, any ideal of the form $I\rtimes_\beta G$ is $\widehat{G}$-invariant. Conversely, let $J\triangleleft B$ be $\widehat{G}$-invariant. Then $I:= J\cap B\triangleleft B$ is a $G$-invariant ideal, since the action $\beta$ is implemented by unitaries in $B\rtimes_\beta G$. Then, $I\subseteq J$ implies $I\rtimes_\beta G\subseteq J$.

For the reverse inclusion, as in Lemma \ref{lem:ideal_facts}.(ii), there is a canonical onto $\ast$-homomorphism $\pi:B\rtimes_\beta G\to (B/I)\rtimes_{\tilde{\beta}} G$ with kernel $I\rtimes_\beta G$. Given $x\in J$, because $J$ is closed and $\widehat{G}$-invariant $E_{\widehat{G}}(x^\ast x)\in J\cap B=I$ and hence $\pi(E_{\widehat{G}}(x^\ast x))=0$. Since $\pi$ is $\widehat{G}$-equivariant, we find
\[
0=
\pi(E_{\widehat{G}}(x^\ast x)) =
E_{\widehat{G}}(\pi(x^\ast x)). 
\]
As the expectation $E_{\widehat{G}}$ is faithful, $\pi(x)=0$ and $x\in \ker \pi=I\rtimes_\beta G$. Therefore $J=I\rtimes_\beta G$.
\end{proof}
\end{lem}

We can now identify the Shilov ideal in $C^\ast(\varphi(A\nctimes P))=p_A(B\rtimes G)p_A$, for any minimal Nica-covariant automorphic dilation $(B,\beta,G)$.

\begin{thm}\label{thm:envelope}
Let $(G,P)$ be a lattice ordered abelian group, and let $(A,\alpha,P)$ be a unital C*-dynamical system over $P$. Suppose $(B,\beta,G)$ is any minimal Nica-covariant automorphic dilation of $(A,\alpha,P)$, with $\ast$-embedding $\iota:A\to B$. Then, there is a unique maximal $\beta$-invariant $A$-boundary ideal $I\triangleleft B$. Further, if $p_A=\iota(1_A)\in B$ and $\varphi=\iota\times up_A:A\nctimes_\alpha P\to B\rtimes_\beta G$ is the completely isometric embedding from Proposition \ref{prop:automorphic_dilation_embedding}, then
\[
p_A(I\rtimes_\beta G)p_A \triangleleft 
p_A(B\rtimes_\beta G)p_A =
C^\ast(\varphi(A\nctimes_\alpha P))
\]
is the Shilov ideal for $A\nctimes_\alpha P$. Consequently
\[
C_e^\ast(A\nctimes_\alpha P)\cong
(p_A+I)\left(\frac{B}{I}\rtimes_{\tilde{\beta}} G\right)(p_A+I)
\]
is a full corner of a crossed product.
\end{thm}

\begin{proof}
Let $\varphi=\iota\times up_A:A\nctimes_\alpha P\to B\rtimes_\beta G$ be the completely isometric representation from Proposition \ref{prop:automorphic_dilation_embedding}. Let $J\triangleleft p_A(B\rtimes_\beta G)p_A$ be the Shilov ideal for $A\nctimes_\alpha P$. Since $\varphi(A\nctimes_\alpha P)=\cspn\{u_p\iota(a)\mid p\in P, a\in A\}$ is invariant under the gauge action of $\widehat{G}$, it follows that $J$ is also $\widehat{G}$-invariant. Let $K=\overline{(B\rtimes_\beta G)J(B\rtimes_\beta G)}$ be the ideal $J$ generates in the entire crossed product $B\rtimes_\beta G$. Since $J$ is $\widehat{G}$-invariant, so too is $K$. By Lemma \ref{lem:invariant_ideal}, we have $K=I\rtimes_\beta G$ for some $\beta$-invariant $I\triangleleft B$. By Lemma \ref{lem:ideal_facts}.(i), we find
\[
J=
p_AKp_A=
p_A(I\rtimes_\beta G)p_A.
\]
Because $\iota(A)\subseteq p_A(B\rtimes_\beta G)p_A$, we also find
\[
I\cap \iota(A)=
K\cap \iota(A) =
p_A(K\cap \iota(A))p_A =
J\cap \iota(A) = \{0\},
\]
since $J$ does not intersect $\varphi(A\nctimes_\alpha P)\supseteq \iota(A)$. Therefore $I$ is a $\beta$-invariant boundary ideal. By Lemma \ref{lem:ideal_facts}, we have a canonical isomorphism
\[
C_e^\ast(A\nctimes_\alpha P)\cong
\frac{p_A(B\rtimes_\beta G)p_A}{p_A(I\rtimes_\beta G)p_A} \cong
(p_A+I)\left(\frac{B}{I}\rtimes_{\tilde{\beta}}G\right)(p_A+I).
\]

To see that $I$ is the unique maximal such ideal, suppose that $R\triangleleft B$ is any $\beta$-invariant $A$-boundary ideal. Then $p_A(R\rtimes_\beta G)p_A\triangleleft p_A(B\rtimes_\beta G)p_A$. By Lemma \ref{lem:ideal_facts} again,
\[
\frac{p_A(B\rtimes_\beta G)p_A}{p_A(R\rtimes_\beta G)p_A} \cong
(p_A+R)\left(\frac{B}{R}\rtimes_{\tilde{\beta}} G\right)(p_A+R).
\]
Then by Remark \ref{rem:quotient_dilation}, $(B/R,\tilde{\beta},G)$ is a minimal Nica covariant automorphic dilation. By Proposition \ref{prop:automorphic_dilation_embedding}, $(p_A+R)((B/R)\rtimes_\beta G)(p_A+R)$ is a C*-cover for $A\nctimes P$. By definition of the C*-envelope, there is an onto $\ast$-homomorphism
\[
\frac{p_A(B\rtimes_\beta G)p_A}{p_A(R\rtimes_\beta G)p_A}\cong
(p_A+R)\left(\frac{B}{R}\rtimes_{\tilde{\beta}} G\right)(p_A+R) \to
C_e^\ast(A\nctimes_\alpha P)\cong
\frac{p_A(B\rtimes_\beta G)p_A}{p_A(I\rtimes_\beta G)p_A},
\]
which fixes $A\nctimes P$. It follows that $p_A(R\rtimes_\beta G)p_A\subseteq p_A(I\rtimes_\beta G)p_A$. Upon intersecting with $B$, in which $p_A$ is central, we find $p_AR\subseteq p_AI$. Since $R$ and $I$ are $\beta$-invariant, and $(\beta_g(p_A))_{g\in G}$ is an approximate identity in $B$, by Lemma \ref{lem:p_A_approx_identity}, it follows that $R\subseteq I$. Indeed, for $x\in R$,
\[
x\beta_g(p_A) =
\beta_g(\beta_{-g}(x)p_A)
\]
lies in $\beta_g(Rp_A)\subseteq  \beta_g(I)\subseteq I$, and converges as a net indexed by $g\in G$ to $x\in R$.
\end{proof}

\begin{cor}\label{cor:no_boundary_ideals}
Suppose $(A,\alpha,P)$ is a unital C*-dynamical system over a lattice ordered abelian group $(G,P)$. If $(B,\beta,G)$ is a minimal Nica-covariant automorphic dilation of $(A,\alpha,P)$, then the C*-cover
\[
\varphi:A\nctimes_\alpha P \to p_A(B\rtimes_\beta G)p_A
\]
is a C*-envelope if and only if $B$ contains no nontrivial $\beta$-invariant $A$-boundary ideals.
\end{cor}

\begin{cor}\label{cor:envelope_full_corner}
Suppose that $(A,\alpha,P)$ is a unital C*-dynamical system, where $(G,P)$ is a lattice ordered abelian group. The C*-envelope $C_e^\ast(A\nctimes_\alpha P)$ is a full corner of a crossed product of a minimal Nica-covariant automorphic dilation of $(A,\alpha,P)$.
\end{cor}

\begin{proof}
To apply Theorem \ref{thm:envelope}, it is enough to note that $(A,\alpha,P)$ has at least one minimal Nica-covariant automorphic dilation. The product dilation $(B,\beta,G)$ from Definition \ref{def:product_dilation} suffices. Then
\[
C_e^\ast(A\nctimes_\alpha P)\cong
(p_A+I)\left(\frac{B}{I}\rtimes G\right)(p_A+I),
\]
and by Remark \ref{rem:quotient_dilation}, $(B/I,\tilde{\beta},G)$ is itself a minimal Nica-covariant automorphic dilation.
\end{proof}

\begin{rem}\label{rem:commutative_envelope}
Note that when $A\cong C(X)$ is a commutative C*-algebra, the product dilation $B\subseteq \prod_{g\in G} A$ is also commutative. Consequently the minimal Nica-covariant automorphic dilation in Corollary \ref{cor:envelope_full_corner} is a quotient of the product dilation, and hence also commutative.
\end{rem}

Corollary \ref{cor:envelope_full_corner} extends even to nonunital systems. To show this, we use essentially the same unitization technique as in \cite{davidson_semicrossed_2017}*{Section 4.3}.

\begin{cor}\label{cor:non_unital}
Let $(G,P)$ be a lattice ordered abelian group, and $(A,\alpha,P)$ a (possibly nonunital) C*-dynamical system. The C*-envelope of $A\nctimes_\alpha P$ is a full corner of a crossed product associated to a minimal Nica-covariant automorphic dilation of $(A,\alpha,P)$.
\end{cor}

\begin{proof}
Form the unitization $\tilde{A}:=A\oplus \C 1_{\tilde{A}}$, even if $A$ is unital. Then we get a unital C*-dynamical system $(\tilde{A},\tilde{\alpha},P)$ by setting $\tilde{\alpha}(a+\lambda 1_{\tilde{A}}):=\alpha(a)+\lambda 1_{\tilde{A}}$ for $a\in A$ and $\lambda \in \C$. Let $(\tilde{B},\tilde{\beta},G)$ be the product dilation of $(\tilde{A},\tilde{\alpha},P)$, with inclusion $\iota:\tilde{A}\to \tilde{B}$.

Now define
\[
B:=\overline{\bigcup_{g\in G}\tilde{\beta}\iota(A)} =
\overline{\sum_{g\in G}\tilde{\beta}_g\iota(A)}\subseteq \tilde{B},
\]
and set $\beta:=\tilde{\beta}\vert_B$. Since $A$ is an $\tilde{\alpha}$-invariant ideal in $\tilde{A}$, it follows that $B$ is a $\tilde{\beta}$-invariant ideal in $\tilde{B}$. By definition, $(B,\beta,G)$ is just the product dilation for $(A,\alpha,P)$. Using the faithfulness of the associated Fock or left regular representations, one can prove that $A\nctimes_\alpha P$ embeds completely isometrically into $\tilde{A}\nctimes_{\tilde{\alpha}} P$, and that $B\rtimes_\beta G$ embeds into $\tilde{B}\rtimes_{\tilde{\beta}} G$. Let $p_{\tilde{A}}:=\iota(1_{\tilde{A}})$, and let $\varphi=\iota\times up_{\tilde{A}}:\tilde{A}\nctimes P\to p_{\tilde{A}}(\tilde{B}\rtimes G)p_{\tilde{A}}$ be the completely isometric embedding from Proposition \ref{prop:automorphic_dilation_embedding}. A similar argument as in the proof of Proposition \ref{prop:automorphic_dilation_embedding} proves that
\[
C^\ast(\varphi(A\nctimes_\alpha P)) =
p_{\tilde{A}}(B\rtimes_\beta G)p_{\tilde{A}}.
\]

Thus, the corner $p_{\tilde{A}}(B\rtimes_\beta G)p_{\tilde{A}}$ is a C*-cover for $A\nctimes P$. This is a full corner of $B\rtimes_\beta G$, because it contains $\iota(A)\subseteq B$, which generates $B$ as a $G$-C*-algebra. Let $J$ be the Shilov ideal for $A\nctimes P$ in $p_{\tilde{A}} (B\rtimes G)p_{\tilde{A}}$. Observe that Lemma \ref{lem:ideal_facts}.(i) holds even in the setting where $C\triangleleft \tilde{C}$ is an ideal in some larger C*-algebra $\tilde{C}$, and the projection $p$ lies in $\tilde{C}$. In particular, using $B\rtimes G\triangleleft \tilde{B}\rtimes G$, and the projection $p_A\in \tilde{B}\rtimes G$, the proof given for Theorem \ref{thm:envelope} applies verbatim to show that $J=p_{\tilde{A}}(I\rtimes G)p_{\tilde{A}}$ for some unique maximal $\beta$-invariant $A$-boundary ideal $I$ in $B$.

 Let $\tilde{I}$ be the unique maximal $\tilde{\beta}$-invariant $\tilde{A}$-boundary ideal in $\tilde{B}$. By construction of the product dilation, we have $B\cap \iota(\tilde{A})=\iota(A)$. It follows that a $\beta$-invariant ideal in $B$ is an $A$-boundary ideal if and only if it is also an $\tilde{A}$-boundary ideal. Therefore $I=\tilde{I}\cap B$. Identifying $B/I=B/(\tilde{I}\cap B)\subseteq \tilde{B}/\tilde{I}$, applying Remark \ref{rem:quotient_dilation} and Lemma \ref{lem:ideal_facts}.(ii) shows that
 \[
 C_e^\ast(A\nctimes P)\cong
 (p_{\tilde{A}}+\tilde{I})\left(\frac{B}{I}\rtimes G\right)(p_{\tilde{A}}+\tilde{I})
 \]
 is a full corner of a crossed product associated to a minimal Nica-covariant automorphic dilation.
\end{proof}

When $(A,\alpha,P)$ is an injective C*-dynamical system, we recover a known result that the C*-envelope of $A\nctimes P$ is a crossed product of a certain minimal automorphic extension of $A$.

\begin{prop}\label{prop:envelope_injective}\cite{davidson_semicrossed_2017}*{Theorem 4.2.12}
Let $(G,P)$ be a lattice ordered abelian group, and $(A,\alpha,P)$ an injective unital C*-dynamical system. Then
\[
C_e^\ast(A\nctimes_\alpha P) \cong
\tilde{A}\rtimes_{\tilde{\alpha}} G,
\]
where $(\tilde{A},\tilde{\alpha},G)$ is an automorphic C*-dynamical system (unique up to equivariant $\ast$-isomorphism) satisfying $A\subseteq \tilde{A}$ and $\tilde{\alpha}_p\vert_A=\alpha_p$ for $p\in P$.
\end{prop}

\begin{proof}
Let $(B,\beta,G)$ be the product dilation for $(A,\alpha,P)$. Then $B\subseteq \prod_GA$. Let
\[
c_0(G,A):=
\{x\in \prod_G A\;\mid\; \lim_g \|x_g\|=0\}\triangleleft \prod_G A.
\]
Here, by writing ``$\lim_{g\in G}$", we are considering $G$ as a directed set in its ordering induced by $P$, and thinking of $G$-tuples as nets. We will show that
\[
I:=
B\cap c_0(G,A) \subset B
\]
is the unique maximal $\beta$-invariant $A$-boundary ideal in $B$. It is easy to check it is a $\beta$-invariant ideal. Because the action $\alpha$ is injective, each $\alpha_p$ is isometric. So, if $a\in \iota^{-1}(I)$, we have
\[
0 =
\lim_{g\in G} \|\iota(a)_g\| =
\lim_{p\in P} \|\alpha_p(a)\| =
\lim_{p\in P} \|a\|=\|a\|,
\]
hence $a=0$. Note that the second equality holds because $P$ is a cofinal subset of $G$. This proves $\iota(A)\cap I=\{0\}$, so $I$ is a $\beta$-invariant $A$-boundary ideal.

Suppose $J\triangleleft B$ is any other $\beta$-invariant $A$-boundary ideal. Let $x\in J\subseteq \prod_GA$. Let $\varepsilon>0$. Because $B$ is a minimal dilation, we can choose an element of the form
\[
y =
\sum_{g\in F}\beta_{-g}\iota(a_g),
\]
where $F\subseteq G$ is finite, and $a_g\in A$, and $\|y-x\|<\varepsilon$. Since $J$ is a $\beta$-invariant ideal, $p_A\beta_{\vee F}(x)$ is in $J$. However, since $(B,\beta,G)$ is a Nica-covariant automorphic dilation,
\begin{align*}
p_A\beta_{\vee F}(y) &=
\sum_{g\in F}p_A\beta_{\vee F-g}\iota(a_g) \\ &=
\sum_{g\in F} \iota\alpha_{\vee F-g}(a_g) \\ &=
\iota\left(\sum_{g\in F}\alpha_{\vee F-g}(a_g)\right)
\end{align*}
is in $\iota(A)$. Since $J$ is an $A$-boundary ideal, the projection $A\to B\to B/J$ is injective, and so isometric. Therefore
\begin{align*}
\|p_A\beta_{\vee F}(y)\| &=
\|p_A\beta_{\vee F}(y)+J\| \\ &\le
\|p_A\beta_{\vee F}(y)-p_A\beta_{\vee F}(x)\| \\ &\le \|y-x\|<\varepsilon.
\end{align*}
Since $[p_A\beta_{\vee F}(y)]_p=y_{p+\vee F}$ for $p\in P$, it follows that $g \ge \vee F$ implies $\|y_g\|<\varepsilon$, and also $\|x_g\|\le\|y_g\|+\|x-y\|<2\varepsilon$. This proves that $x\in c_0(G,A)$, so $J\subseteq I$, and $I$ is the unique maximal $\beta$-invariant $A$-boundary ideal in $B$.

By Theorem \ref{thm:envelope},
\[
C_e^\ast(A\nctimes_\alpha P)\cong
(p_A+I)\left(\frac{B}{I}\rtimes_{\tilde{\beta}} G\right)(p_A+I).
\]
However, $p_A=1_{\prod A}$ modulo $c_0(G,A)$, because $p\ge 0$ implies $[p_A]_p=1$. It follows that $p_A+I$ is a two-sided identity $1_{B/I}$, and the C*-envelope is just the crossed product $(B/I)\rtimes_{\tilde{\beta}} G$. By Remark \ref{rem:quotient_dilation}, Nica-covariance of the dilation $(B/I,\tilde{\beta},G)$, with unital embedding $\eta=q_I\iota:A\to B\to B/I$, implies that, for $p\in P$ and $a\in A$, 
\[
\beta_p\eta(a)=
(p_A+I)\beta_p(\eta(a)) =
\eta(\alpha_p(a)).
\]
So, $\tilde{\beta}_p\eta=\eta\alpha_p$, which when we identify $A\cong \eta(A)\subseteq B/I$, reads $\tilde{\beta}_p\vert_A=\alpha_p$. Since the automorphic dilation $(B/I,\tilde{\beta},G)$ is minimal, it also follows easily that
\[
\frac{B}{I}=
\overline{\bigcup_{p\in P}\tilde{\beta}_{-p}\eta(A)}.
\]
Thus $(B/I,\tilde{\beta},G)$ is a minimal automorphic extension of $(A,\alpha,P)$. Such an extension is unique up to an equivariant isomorphism fixing $A$, since if
\[
\tilde{A}=\overline{\bigcup_{p\in P}\tilde{\alpha}_{-p}(A)}\supseteq A,
\]
with $G$-action $\tilde{\alpha}$ extending $\alpha$, the map $\tilde{\beta}_{-p}\eta(a)\mapsto \tilde{\alpha}_{-p}(a)$ extends to an equivariant $\ast$-isomorphism $B/I\cong \tilde{A}$.
\end{proof}

In the proof of Proposition \ref{prop:envelope_injective}, we showed the maximal $\beta$-invariant $A$-boundary ideal was $B\cap c_0(G,A)$. In the case $(G,P)=(\Z,\Z_+)$, this result generalizes readily to the non-injective case.

Recall that if $A$ is a C*-algebra and $I
\triangleleft A$ is an ideal, then
\[
I^\perp :=
\{a\in A\mid b\in I\implies ab=0\}\subseteq A
\]
is also an ideal, and satisfies $I\cap I^\perp =\{0\}$. If $\pi:A\to B$ is a $\ast$-homomorphism and $a\in (\ker \pi)^\perp$, then
\[
\|\pi(a)\|=
\|a+\ker \pi\|=
\|a+(\ker \pi)^\perp \cap \ker \pi\|=
\|a\|.
\]
This shows $\pi \vert_{(\ker \pi)^\perp}$ is always isometric.

\begin{prop}\label{prop:envelope_Z}
Let $(A,\alpha,\Z_+)$ be a unital C*-dynamical system over $\Z_+$, and let $(B,\beta,\Z)$ be its product dilation. The unique maximal $\beta$-invariant $A$-boundary ideal for $\iota(A)$ in $B$ is
\[
I=
B\cap c_0(\Z,(\ker \alpha)^\perp).
\]
Consequently,
\[
C_e^\ast(A\nctimes_\alpha P)\cong
(p_A+I)\left(\frac{B}{I}\rtimes_{\tilde{\beta}} G\right)(p_A+I).
\]
\end{prop}

\begin{proof}
Because $(\ker \alpha)^\perp$ is an ideal in $A$, it follows easily that $I$ is a $\beta$-invariant ideal in $B$. Suppose $a\in A$ with $\iota(a)\in I\subseteq c_0(\Z,(\ker \alpha)^\perp)$. Then, each $\alpha^n(a)\in (\ker \alpha)^\perp$. Because $\alpha$ is isometric on $(\ker \alpha)^\perp$, one sees that $\|\alpha^n(a)\|=\|a\|$ by an easy induction on $n$, and so
\[
0=
\lim_{n\to \infty}\|\iota(a)_n\|=
\lim_{n\to \infty}\|\alpha^n(a)\|= 
\lim_{n\to\infty} \|a\| = \|a\|.
\]
Thus $\iota(A)\cap I=\{0\}$.

Suppose $J\triangleleft B$ is any $\beta$-invariant boundary ideal for $A$. The same argument as in the proof of Proposition \ref{prop:envelope_injective} shows that all tuples in $J$ vanish at $+\infty$. So, it suffices to let $x\in J$ and prove each $x_g\in (\ker \alpha)^\perp$. If $b\in \ker \alpha$, then
\[
\iota(b)=
(\ldots,0,0,b,0,0,\ldots).
\]
So,
\[
\beta_g(x)\iota(b) =
(\ldots,0,0,x_gb,0,0\ldots) =
\iota(x_gb) \in \iota(A)\cap J=\{0\}.
\]
Since $\iota$ is injective, $x_gb=0$, and this proves each $x_g\in (\ker \alpha)^\perp$. So, $J\subseteq I$. Therefore $I$ is the unique maximal $\beta$-invariant $A$-boundary ideal in $B$, and Theorem \ref{thm:envelope} applies.
\end{proof}

\begin{eg}\label{eg:Z_doesnt_generalize}
Proposition \ref{prop:envelope_Z} does not generalize so readily to the case $P=\Z_+^n$. Consider the unital C*-dynamical system $(A,\alpha,\Z_+^2)$, where $A=\C^3$ and the action is determined by generators by
\begin{align*}
\alpha_1(a,b,c)&=
(a,c,c),\\
\alpha_2(a,b,c) &=
(c,b,c).
\end{align*}
This is the unitization of the nonunital system $(\C\oplus \C,\alpha^0,\Z_+^2)$, where $\alpha^0_1(a,b)=(a,0)$ and $\alpha^0_2(a,b)=(0,b)$. Reviewing Proposition \ref{prop:envelope_Z}, we might expect
\[
B\cap c_0(\Z,R_\alpha^\perp)=
\{b\in B\mid b_g\in R_\alpha^\perp \text{ and }\lim_{g\in \Z^2} b_g=0\},
\]
for $R_\alpha=\ker \alpha_1\cap \ker \alpha_2$, to be the unique maximal $\beta$-invariant $A$-boundary ideal. However, this fails to even be a boundary ideal, since here $R_\alpha=\{0\}$, and for any element $x=(a,b,0)\in A$, the tuple
\begin{align*}
\iota(x) =
\begin{pmatrix}
    \ddots \\
    & 0 & 0 & 0 & 0 & \cdots \\
    & 0 & (a,b,0) & (0,b,0) & (0,b,0) & \cdots \\
    & 0 & (a,0,0) & 0 & 0 & \cdots \\
    & 0 & (a,0,0) & 0 & \ddots \\
    & \vdots  & \vdots & \vdots  &
\end{pmatrix}
\end{align*}
lies in $A\cap c_0(\Z,R_\alpha^\perp)=\iota(A)\cap c_0(\Z,A)$. A correct description of the Shilov ideal in the case $P=\Z^n_+$ is more complicated, and follows as described in Section \ref{sec:Shilov_ideal}. See Section \ref{sec:Zn} for more discussion in the case $P=\Z^n_+$. 
\end{eg}

\section{Explicit Computation of the Shilov Ideal}\label{sec:Shilov_ideal}

Throughout this section, let $(G,P)$ be a lattice ordered abelian group, and let $(A,\alpha,P)$ be a unital C*-dynamical system. Also, let $(B,\beta,G)$ be the product dilation for $(A,\alpha,P)$, with inclusion $\iota:A\to B\subseteq \prod_G A$, as in Definition \ref{def:product_dilation}. By Theorem \ref{thm:envelope}, $B$ contains a unique maximal ideal $I$ which is both $\beta$-invariant and an $A$-boundary ideal (does not intersect $\iota(A)$). In this section, we will explicitly describe $I$. The following construction of $I$ was inspired both by the construction in \cite{davidson_semicrossed_2017}*{Section 4.3}, and the construction of Sehnem's covariance algebra in \cite{sehnem_c*-algebras_2019}*{Section 3.1}.

\begin{defn}\label{def:product_Shilov_ideal}
Define the following ideals.
\begin{itemize}
\item[(1)] Given a finite subset $F\subseteq G$, let
\[
K_F:=\bigcap_{\substack{g\in F \\ g\not\le 0}}\ker \alpha_{g\vee 0}\triangleleft A
\]
be the ideal of elements vanishing under the action of any strictly positive part of an element in $F$. Here we take the convention that the empty intersection yields $K_F=A$.
\item[(2)] For $F\subseteq G$ finite, let
\[
J_F:=
K_F^\perp \triangleleft A
\]
be the annihilator of $K_F$.
\item[(3)] For $F\subseteq G$ finite, define
\[
I_F:=
\left\{b\in B\;\middle|\; b_g\in J_{F-g} \text{ for all }g\in G\right\} \triangleleft B.
\]
\item[(4)] Finally, set
\[
I:=
\overline{\bigcup_{F\subseteq G\text{ finite}}I_F}\triangleleft B.
\]
\end{itemize}
\end{defn}

It is straightforward to check that if $F\subseteq F'$ are finite subsets of $G$, then $K_F\supseteq K_{F'}$, and hence $J_F\subseteq J_{F'}$. Consequently $I_F\subseteq I_{F'}$, so $\{I_F\mid F\subseteq G\text{ finite}\}$ is a directed system of ideals, and so $I$ is indeed an ideal in $B$. Further, it's just as straightforward to show that for any $g\in G$, and any finite $F\subseteq G$, that
\[
\beta_g(I_F)=I_{F-g}.
\]
It follows that $I=\overline{\bigcup_F I_F}$ is a $\beta$-invariant ideal.

\begin{thm}\label{thm:I_is_Shilov_ideal}
Let $(A,\alpha,P)$ be a unital C*-dynamical system over a lattice ordered abelian group $(G,P)$, with product dilation $(B,\beta,G)$.The ideal $I\triangleleft B$ from Definition \ref{def:product_Shilov_ideal} is the unique maximal $\beta$-invariant $A$-boundary ideal in the product dilation $(B,\beta,G)$. Consequently,
\[
C_e^\ast(A\nctimes_\alpha P)\cong
(p_A+I)\left(\frac{B}{I}\rtimes_{\tilde{\beta}}G\right)(p_A+I)
\]
is a full corner of a crossed product.
\end{thm}

For clarity, we break the proof of Theorem \ref{thm:I_is_Shilov_ideal} into lemmas. Our first lemma is a verification that $I$ is indeed a $\beta$-invariant $A$-boundary ideal.

\begin{lem}\label{lem:I_boundary_ideal}
The ideal $I$ satisfies $\iota(A)\cap I=\{0\}$.
\end{lem}
\begin{proof}
Since $I=\overline{\bigcup_F I_F}$ is an inductive union of ideals, it suffices to prove $I_F\cap \iota(A)=\{0\}$ for every finite $F\subseteq G$. Suppose for a contradiction that there is some finite $F_0\subseteq G$, and some nonzero $a_0\in A\setminus \{0\}$ with $\iota(a_0)\in I_{F_0}$. By definition of $I$,
\[
0\ne a_0=\iota(a_0)_0\in J_{F_0}=K_{F_0}^\perp =\bigg(\bigcap_{\substack{g\in F_0\\ g\not\le 0}}\ker \alpha_{g\vee 0}\bigg)^\perp.
\]
Since $a_0\ne 0$ and $K_{F_0}\cap K_{F_0}^\perp =\{0\}$, the element $a_0$ is not in $K_{F_0}$. So, there is a $g\in F_0$ with $g\vee 0>0$ and $\alpha_{g\vee 0}(a_0)\ne 0$.

Set $a_1=\alpha_{g\vee 0}(a_0)\ne 0$. Since $\iota(a_0)\in I_{F_0}$, it will follow that $\iota(a_1)=\iota\alpha_{g\vee 0}(a_0)$ lies in $I_{F_1}$, where
\[
F_1:=
\{h-g\vee 0\mid h\in F_0, h\not \le g\} \subset F_0-g\vee 0.
\]
Because $g-g\vee 0$ lies in $(F_0-g\vee 0)\setminus F_1$, we also have $|F_1|<|F_0|$ strictly. But then, because $a_1\ne 0$ and $\iota(a_1)\in I_{F_1}$, we may repeat the same argument to find a nonzero $a_2\in A$ and an $F_2\subseteq G$, with $\iota(a_2)\in I_{F_2}$, and $|F_2|<|F_1|$. Continuing recursively, we find an infinite sequence
\[
|F_0|>|F_1|>|F_2|>\cdots
\]
of finite subsets of $G$, and each $I_{F_n}\cap \iota(A)\ne \{0\}$. This is absurd, since eventually such a sequence must terminate at $\emptyset$, and $I_\emptyset=\{0\}$. This proves $I_F\cap \iota(A)=\{0\}$.

To prove that $\iota(a_1)\in I_{F_1}$ as needed in the paragraph above, it suffices to note that for any $p\in P$, that
\begin{align*}
K_{F_0-g\vee 0 -p} &=
\bigcap_{\substack{h\in F_0\\ h-g\vee 0-p\not\le 0}}\ker \alpha_{(h-g \vee 0 - p)\vee 0}\\ &\supseteq
\bigcap_{\substack{k\in F_1\\ k-p\not\le 0}}\ker \alpha_{(k-p)\vee 0} =K_{F_1-p}.
\end{align*}
This is because if $h\in F_0$ with $h-g\vee 0-p\not \le 0$, then
\begin{align*}
0&<
(h-g\vee 0-p)\vee 0 \\ &\le
(h-g\vee 0)\vee 0 \\&=
h\vee g\vee 0-g\vee 0.
\end{align*}
Therefore $h\vee g\ne g$ and $h\not \le g$, so in fact $h-g\vee 0\in F_1$. Knowing this, for any $p\in P$, we have
\[
\iota(a_1)_p=
\alpha_{g\vee 0+p}(a_0)\in
I_{F_0-g\vee 0-p}\subseteq I_{F_1-p},
\]
proving $\iota(a_1)\in I_{F_1}$.
\end{proof}

To prove Theorem \ref{thm:I_is_Shilov_ideal}, it will be very helpful to identify $B$ as a direct limit over certain finite subsets of $G$.

\begin{defn}\label{def:grid}\cite{davidson_semicrossed_2017}*{Section 4.2}
Let $(G,P)$ be a lattice ordered abelian group. A subset $F\subseteq G$ is a \textbf{grid} if $F$ is finite and closed under $\vee$.
\end{defn}

Since any finite subset $F$ of $G$ is contained in a grid, found by appending all joins of finite subsets of $F$, the set of all grids in $G$ is directed under inclusion and $G=\bigcup\{F\subseteq G \text{ grid}\}$.

\begin{lem}\label{lem:product_direct_limit}
The product dilation $B$ is an internal direct limit
\[
B=
\overline{\bigcup_{F\subseteq G\text{ grid}}B_F}
\]
of C*-subalgebras
\[
B_F:=
\sum_{g\in F}\beta_{-g}\iota(A).
\]
\end{lem}
\begin{proof}
In fact, for any minimal Nica-covariant automorphic dilation $(B,\beta,G)$ (Definition \ref{def:automorphic_dilation}), we have
\[
B=
\overline{\sum_{g\in G}\beta_g\iota(A)}=
\overline{\bigcup_{F\subseteq G\text{ grid}} B_F}.
\]
The subspaces $B_F$ are always $\ast$-subalgebras, because all maps involved are $\ast$-linear, and the multiplication formula
\[
\beta_{-g}\iota(a)\beta_{-h}\iota(b) =
\beta_{-(g\vee h)}\iota(\alpha_{g\vee h-g}(a)\alpha_{h-g\vee h}(b)),
\]
for $g,h\in G$ and $a,b\in A$, implies that $B_F$ is multiplicatively closed when $F$ is $\vee$-closed.

So we need only show each $B_F$ is norm closed, and this is where we use the construction of the product dilation. We will use induction on $|F|$. Certainly $B_\emptyset =\{0\}$ is closed. Fix a nonempty grid $F\subseteq G$, and suppose whenever $F'\subseteq G$ is a grid with $|F'|< |F|$, that $B_{F'}\subseteq B$ is closed. Choose a convergent sequence $x_n\in B_F$, and write
\[
x_n=
\sum_{g\in F}\beta_{-g}\iota(a_n^g),\quad
 a_n^g\in A.
\]
Since $F$ is finite, $F$ contains a minimal element $g_0$. By minimality of $g_0$, we have $[x_n]_{g_0}=a_n^{g_0}$. Then,
\[
\|a_n^{g_0}-a_m^{g_0}\|\le
\|x_n-x_m\|,
\]
so the sequence $a_n^{g_0}$ is Cauchy, and has a limit $a^{g_0}\in A$. Then,
\[
y_n:=
x_n-\beta_{-g_0}\iota(a_n^{g_0})=
\sum_{g\in F\setminus \{g_0\}} \beta_{-g}\iota(a_n^g)
\]
is a Cauchy sequence in $B_{F\setminus \{g_0\}}$. As $F\setminus \{g_0\}$ is a grid of smaller size then $F$, $y_n$ has a limit $y\in B_{F\setminus \{g_0\}}\subseteq B_F$. But then $x_n=\beta_{-g_0}\iota(a_n^{g_0})+y_n$ converges to
\[
\beta_{-g_0}\iota(a^{g_0})+y\in B_F.
\]
So, $B_F$ is closed, finishing the induction.
\end{proof}

The next lemma offloads a technical step in the proof of Theorem \ref{thm:I_is_Shilov_ideal}. The point is that, when $F\subseteq G$ is any grid, and $a\in A$, the entries of the tuple $\iota(a)\in B\subseteq \prod_G A$ are realized by an element of $B_F$ for ``large enough" $g\in G$.

\begin{lem}\label{lem:grid_entries}
Let $F\subseteq G$ be a grid. Then there are integers $c_g\in \Z$, such that whenever $a\in A$, and $h\ge g$ for at least one element $g\in F$, we have
\[
\iota(a)_h=
\alpha_h(a)=
\left[\sum_{g\in F} c_g\cdot \beta_{-g}\iota\alpha_g(a)\right]_h.
\]
\end{lem}
\begin{proof}
It will be enough to find integers $c_g$ such that, for any $g\in F$,
\[
c_g=
1-\sum_{\substack{h\in F\\ h<g}}c_h.
\]
We can build such $c_g$ recursively. Choose some minimal element $g_0\in F$, and set $c_{g_0}:=1$. Assuming inductively that $c_{g_0},\ldots,c_{g_n}$ have been defined, so that each $g_k$ is minimal in $F\setminus \{g_0,\ldots,g_{k-1}\}$, we can set
\[
c_{g_{n+1}}:=
1-\sum_{\substack{h\in F\\ h<g_{n+1}}}c_h.
\]
Note that if $h\in F$ and $h<g_{n+1}$, then minimality of $g_{n+1}$ implies that $h$ appears in the list $\{g_0,\ldots,g_n\}$, so $c_h$ is defined.

Completing the inductive construction, we find integers $c_g$, $g\in F$, with
\[
\sum_{\substack{h\in F\\ h\le g}}c_h=1
\]
for any $g\in F$. Now, let
\[
x:=
\sum_{g\in F}c_g\cdot \beta_{-g}\iota\alpha_g(a)\in B_F.
\]
Then, whenever $h$ dominates at least one element of $F$,
\begin{align*}
x_h &=
\sum_{g\in F}c_g\, [\iota\alpha_g(a)]_{h-g} \\ &=
\bigg(\sum_{\substack{g\in F\\ g\le h}}c_g\bigg)\alpha_h(a).
\end{align*}
Because $F$ is $\vee$-closed, $U:=\{g\in F\mid g\le h\}$ is equal to $\{g\in F\mid g\le \vee U\}$. So,
\[
\sum_{\substack{g\in F\\g\le h}}c_g=
\sum_{\substack{g\in F\\g\le \vee U}}c_g =1.
\]
This shows that $[x]_h=\alpha_h(a)=\iota(a)_h$. Otherwise, $h$ dominates no element of $F$ and $[\beta_{-g}\iota\alpha_g(a)]_h=[\iota\alpha_g(a)]_{h-g}=0$ for each $g\in F$, so $[x]_h=0$.
\end{proof}

\begin{proof}[Proof of Theorem \ref{thm:I_is_Shilov_ideal}]
From Lemma \ref{lem:I_boundary_ideal}, we already know that $I$ is a $\beta$-invariant $A$-boundary ideal. So, it remains to prove $I$ is maximal among all such ideals. Suppose $R\triangleleft B$ is the unique maximal $\beta$-invariant boundary ideal for $A$, from Theorem \ref{thm:envelope}. Then $I\subseteq R$, but we wish to prove $I=R$. Since $B=\overline{\bigcup\{B_F\mid F\text{ grid}\}}$ is a direct limit (Lemma \ref{lem:product_direct_limit}), and ideals in a C*-algebra are inductive, $R\subseteq I$ if and only if $R\cap B_F\subseteq I\cap B_F$ for every grid $F\subseteq G$.

We will prove $R\cap B_F\subseteq I\cap B_F$ for every grid $F$ by induction on $|F|$. This is immediate when $|F|=0$, since $B_\emptyset=\{0\}$. Suppose now that $|F|>0$ and that if $F'$ is any grid with $|F'|<|F|$, then $R\cap B_{F'}\subseteq I\cap B_{F'}$. Choose any element
\[
x=\sum_{g\in F}\beta_{-g}\iota(a_g)\in R\cap B_F.
\]
Pick a minimal element $g_0\in F$. In fact, since $R$ is $\beta$-invariant we are free to translate so that $g_0=0$ is minimal in $F$. By Lemma \ref{lem:grid_entries} applied for $a_0$ and the grid $F\setminus \{0\}$, we can find an element
\[
y=
\sum_{g\in F\setminus \{0\}}c_g \cdot \beta_{-g}\iota\alpha_g(a_0)
\]
such that if $h\ge g$ for any $g\in F\setminus \{0\}$, then $y_h=\iota(a_0)_h$. Let 
\begin{align*}
z &:=
y+\sum_{g\in F\setminus \{0\}}\beta_{-g}\iota(a_g) \\ &=
\sum_{g\in F\setminus \{0\}}\beta_{-g}\iota(a_g+c_g\,\alpha_g(a_0)),
\end{align*}
so that $z\in B_{F\setminus \{0\}}$. Whenever $h\in G$ dominates a nonzero element of $F$, we have $[y]_h=\iota(a_0)_h$ and so
\[
[z]_h =
\sum_{g\in F} [\beta_{-g}\iota(a_g)]_h =
[x]_h.
\]
Otherwise, if $h\not \ge g$ for all $g\in F\setminus \{0\}$, then any element $w\in B_{F\setminus \{0\}}$ satisfies $w_h=0$. Indeed if $g\in F\setminus \{0\}$ and $d\in A$, then $[\beta_g\iota(d)]_h=[\iota(d)]_{h-g}=0$ because $h-g\not\ge 0$, and any $w\in B_{F\setminus \{0\}}$ is a sum of such terms. So, in this case $[x]_h=\iota(a_0)_h$ and $[z]_h=0$.

We will show that $x-z$ lies in $I_F\subseteq I\subseteq R$. We have
\begin{equation}\label{eq:x-z_entries}
[x-z]_h=
\begin{cases}
    \alpha_h(a_0) & h\ge 0 \text{ and }h\not\ge g\text{ for all }g\in F\setminus \{0\}, \\
    0 & \text{else}.
\end{cases}
\end{equation}
So, suppose $p\in P$, with $p\not\ge g$ for all $g\in F\setminus \{0\}$. Let $b\in K_{F-p}$, so that
\[
b\in \bigcap_{h\in (F-p)\setminus \{0\}} \ker \alpha_{h\vee 0} =
\bigcap_{g\in F\setminus \{0\}} \ker \alpha_{(g-p)\vee 0}.
\]
Then it follows from \eqref{eq:x-z_entries} that
\[
\beta_p(x-z)\iota(b)=
\beta_p(x)\iota(b)=
\iota(\alpha_p(a_0)b),
\]
which, since $x\in R$, and $R$ is $\beta$-invariant, lies in $\iota(A)\cap R=\{0\}$. Since $\iota$ is injective, $\alpha_p(a_0)b=0$. This proves $\alpha_p(a_0)=[x-z]_p\in J_{F-p}=K_{F-p}^\perp$, so indeed $x-z\in I_F\subseteq I\subseteq R$.

As $x$ and $x-z$ are in $R$,
\[
z=x-(x-z)\in R\cap B_{F\setminus \{0\}}.
\]
By inductive hypothesis, since $|F\setminus \{0\}|<|F|$, we conclude $z\in I\cap B_{F\setminus \{0\}}$. Since $x-z\in I$, we find $x=z+(x-z)$ lies in $I\cap B_F$, completing the induction.
\end{proof}

Recall that an ideal $I$ in a C*-algebra $A$ is \textbf{essential} if it intersects every nonzero ideal of $A$, or equivalently if $I^\perp = \{0\}$.

\begin{cor}\label{cor:product_envelope}
Let $(G,P)$ be a lattice ordered abelian group. Let $(A,\alpha,P)$ be a unital C*-dynamical system, with product dilation $(B,\beta,G)$. Then the C*-cover $p_A(B\rtimes_\beta G)p_A$ is the C*-envelope of $A\nctimes_\alpha P$ if and only if for every finite subset $F\subseteq P\setminus \{0\}$,
\[
K_F=\bigcap_{p\in F}\ker \alpha_p
\]
is an essential ideal in $A$.
\begin{proof}
With $I\triangleleft B$ as in Theorem \ref{thm:I_is_Shilov_ideal}, the product dilation yields the C*-envelope if and only if $I=\{0\}$. But by construction, this occurs if and only if each $I_F=\{0\}$, which occurs if and only if each $K_F^\perp=\{0\}$ for any finite subset $F\subseteq G$, or equivalently any finite $F\subseteq P$.
\end{proof}
\end{cor}

\section{The Case \texorpdfstring{$P={\Z_+^n}$}{P=Zn+}.} \label{sec:Zn}

In \cite{davidson_semicrossed_2017}*{Theorem 4.3.7}, Davidson, Fuller, and Kakariadis identify the C*-envelope of a semicrossed product $A\nctimes_\alpha \Z_+^n$ by $\Z_+^n$ as a full corner of a crossed product by $\Z^n$, when $(A,\alpha,\Z^+_n)$ is a C*-dynamical system. In this section, we show that the C*-dynamical system $(B/I,\tilde{\beta},\Z^n)$ from Theorem \ref{thm:I_is_Shilov_ideal} (in the case $(G,P)=(\Z^n,\Z^n_+)$) is $\Z^n$-equivariantly $\ast$-isomorphic to the C*-dynamical system constructed in \cite{davidson_semicrossed_2017}*{Section 4.3}. It follows that the latter system is a minimal automorphic Nica-covariant dilation of $(A,\alpha,\Z^n_+)$ without nontrivial $\Z^n$-invariant $A$-boundary ideals, and we recover \cite{davidson_semicrossed_2017}*{Theorem 4.3.7} from Corollary \ref{cor:no_boundary_ideals}.

We now recall the construction in \cite{davidson_semicrossed_2017}*{Section 4.3}. Since our notation clashes with the notation in that paper, we must introduce new notation. We write the standard generators in $\Z^+_n$ as $\mathbf{1},\ldots,\mathbf{n}$. Given $x=(x_1,\ldots,x_n)\in \Z_+^n$,
\[
\supp(x):=
\{\mathbf{k}\in \{\mathbf{1},\ldots,\mathbf{n}\}\mid x_k>0\}.
\]
If $x,y\in \Z^n_+$, we write $x\perp y$ if $x\wedge y=0$, or equivalently $\supp(x)\cap \supp(y)=\emptyset$. Moreover, let $x^\perp :=\{y\in \Z_+^n\mid y\perp x\}$. Let $(A,\alpha,\Z^+_n)$ be a C*-dynamical system. For $x\in \Z_n^+$, define ideals
\[
Q_x^0:=
\left(\bigcap_{\mathbf{i}\in \supp(x)} \ker \alpha_i\right)^\perp\triangleleft A,
\]
and
\[
Q_x:=
\bigcap_{y\in x^\perp}\alpha_y^{-1}(Q_x^0)\subseteq Q_x^0.
\]
Form the C*-algebra
\[
C:=
\bigoplus_{x\in \Z_+^n}\frac{A}{Q_x}.
\]
Let $q_x:A\to A/Q_x$ be the quotient map. Since $Q_0=\{0\}$, $\eta:= q_0$ is a $\ast$-monomorphism $A\to C$. For convenience, we notationally identify
\[
C =
\overline{\sum_{x\in \Z_+^n} \frac{A}{Q_x}\otimes e_x},
\]
where $e_x$ are formal generators, as in \cite{davidson_semicrossed_2017}*{Section 4.3}. Then $(C,\gamma,\Z^n_+)$ is an injective C*-dynamical system, where the action $\gamma$ is determined on generators by
\[
\gamma_i(q_x(a)\otimes e_x) =
\begin{cases}
    q_x(\alpha_\mathbf{i}(a))\otimes e_x + q_{x+\mathbf{i}}(a)\otimes e_{x+\mathbf{i}} & \mathbf{i}\perp x,\\
    q_{x+\mathbf{i}}(a)\otimes e_{x+\mathbf{i}} & \mathbf{i}\in \supp(x).
\end{cases}
\]
Since $\gamma_i(q_0(a)\otimes e_0) =q_0(\alpha_i(a)\otimes e_0+q_\mathbf{i}(a)\otimes e_\mathbf{i}$ has $0$'th entry $\alpha_i(a)$, the system $(C,\gamma,\Z_+^n)$ dilates $(A,\alpha,\Z_+^n)$ in the same sense as Definition \ref{def:automorphic_dilation}.

Let $(\tilde{C},\tilde{\gamma},\Z^n)$ be the minimal automorphic extension of $(C,\gamma,\Z^n_+)$, from \cite{davidson_semicrossed_2017}*{Theorem 4.2.12}. This C*-dynamical system satisfies
\[
C\subseteq\tilde{C},\quad
\tilde{\gamma}\vert_C=\gamma \quad\text{and}\quad
C =
\overline{\bigcup_{x\in \Z_+^n}\tilde{\gamma}_{-x}(C)}.
\]
Then, $(\tilde{C},\tilde{\gamma},\Z^n)$ is a minimal Nica-covariant automorphic dilation of $(A,\alpha,\Z_+^n)$. Nica-covariance of this dilation is found in \cite{davidson_semicrossed_2017}*{Lemma 4.3.8}. The content of \cite{davidson_semicrossed_2017}*{Theorem 4.3.7} is that the natural map $A\nctimes_\alpha \Z_+^n\to \tilde{C}\rtimes_{\tilde{\gamma}} \Z^n$ is completely isometric, and via this map
\[
C_e^\ast(A\nctimes \Z_+^n)=
p_0(\tilde{C}\rtimes \Z^n)p_0
\]
is a full corner by the projection $p_0=1_A\otimes e_0=\eta(1_A)$.

\begin{prop}\label{prop:Zn_isomorphism}
Let $(A,\alpha,\Z^n_+)$ be a unital C*-dynamical system. Let $(B,\beta,\Z^n)$ be the product dilation (Definition \ref{def:product_dilation}), and $(\tilde{C},\tilde{\gamma},\Z^n)$ be the automorphic dilation defined above. Let $I\triangleleft B$ be the unique maximal $\beta$-invariant $A$-boundary ideal as in Theorem \ref{thm:I_is_Shilov_ideal}. Then there is a $\Z^n$-equivariant $\ast$-isomorphism $B/I\cong \tilde{C}$ that fixes $A$.
\end{prop}

\begin{proof}
Define
\[
\pi:\sum_{x\in \Z^n} \beta_x\iota(A)\to \tilde{C}
\]
to be the unique $\ast$-linear map satisfying
\[
\pi(\beta_x\iota(a))=
\tilde{\gamma}_x\eta(a).
\]
Note that $\pi$ is well defined, because if $F\subseteq G$ is finite, and $a_x\in A$ satisfy
\[
b:=
\sum_{x\in F}\beta_x\iota(a_x)=0,
\]
then each $a_x=0$, so such a representation is well defined. Indeed, if $x_0\in F$ is minimal, then $b_{x_0}=a_{x_0}=0$. By replacing $F$ with $F\setminus \{x_0\}$ and recursing, we eventually find each $a_x=0$. Since both $(B,\beta,\Z^n)$ and $(\tilde{C},\tilde{\gamma},\Z^n)$ are minimal automorphic dilations, $\pi$ is a $\ast$-linear map defined on a dense subalgebra with dense range. By construction $\pi$ is $\Z^n$-equivariant. Nica-covariance of both dilations imply that, if $x,y\in \Z^n$ and $a,b\in A$,
\[
\beta_x\iota(a)\beta_y\iota(b)=
\beta_{x\wedge y}\iota(\alpha_{x-x\wedge y}(a)\alpha_{y-x\wedge y}(b)),
\]
and identically
\[
\tilde{\gamma}_x\eta(a)\tilde{\gamma}_y\eta(b)=
\tilde{\gamma}_{x\wedge y}\eta(\alpha_{x-x\wedge y}(a)\alpha_{y-x\wedge y}(b)).
\]
Extending linearly, it follows that $\pi$ is a $\ast$-homomorphism.

We claim $\pi$ is bounded. Given an element $b= \sum_{x\in F}\beta_g\iota(a_x)\in \sum_g\beta_g\iota(A)$ as above, using \cite{davidson_semicrossed_2017}*{Lemma 4.3.6}, we find
\begin{align*}
\|\pi(b)\| &=
\left\|\sum_{x\in F}\sum_{0\le y\le x} q_y(\alpha_{x-y}(a_x))\otimes e_y\right\| \\ &\le
\sup_{y\in \Z^n_+}\Big\|q_y\Big(\sum_{\substack{x\in F \\ x\ge y}} \alpha_{x-y}(a_x)\Big)\Big\| \\ &\le
\sup_{y\in \Z^n}\Big\|\sum_{\substack{x\in F \\ x\ge y}}\alpha_{x-y}(a_x)\Big\|,
\end{align*}
or upon swapping $y$ with $-y$,
\begin{align*}
\|\pi(b)\| &\le
\sup_{y\in \Z^n} \Big\|\sum_{\substack{x\in F \\ x+y\ge 0}}\alpha_{x+y}(a_x)\Big\|\\ &=
\left\|\sum_{x\in F}\beta_x\iota(a_x)\right\| = \|b\|.
\end{align*}
So, $\pi$ is contractive. Therefore, $\pi$ extends uniquely to an equivariant surjective $\ast$-homomorphism $B\to \tilde{C}$, which we denote by the same symbol. Note also that $\pi\iota=\eta$, so $\pi$ fixes the respective copies of $A$.

The result follows if we can prove $\ker \pi=I$. Since $\pi$ is equivariant and isometric on $\iota(A)\cong \eta(A)\cong A$, $\ker\pi$ is a $\beta$-invariant boundary ideal and so $\ker\pi \subseteq I$, by maximality of $I$. 

To show $I\subseteq \ker \pi$, by inductivity of ideals it suffices to prove
\[
B_F\cap I_H \subseteq \ker \pi,
\]
for any grid $F\subseteq \Z^n$ (Lemma \ref{lem:product_direct_limit}) and any finite subset $H\subseteq \Z^n$. So, it suffices to assume we have an element
\[
b=
\sum_{x\in F}\beta_x\iota(a_x)\in I_H,
\]
where $F\subseteq G$ is finite, and prove $\pi(b)=0$. In fact, since $\pi$ is $\Z^n$-equivariant, and $\beta_g(I_H)=I_{H-g}$, we are free to apply $\beta_g$ for any $g\ge (-\wedge F)\vee (\vee H)$ and so assume $F\subseteq \Z_+^n$ and $H\subseteq -\Z^n_+$. Now, compute
\begin{align*}
\pi(b) &=
\sum_{x\in F}\gamma_x\eta(a_x) \\ &=
\sum_{y\in \Z_+^n}q_y\Big(\sum_{\substack{x\in F\\ x\ge y}}\alpha_{x-y}(a_x)\Big)\otimes e_y \\ &=
\sum_{y\in \Z_+^n}q_y(b_{-y})\otimes e_y.
\end{align*}
So, we must show $b_{-y}\in Q_y$ for all $y\in \Z^n_+$. Suppose that $z\in \Z_+^n$ with $z\perp y$. Then
\begin{align*}
b_{z-y} &=
\sum_{\substack{x\in F\\x+z\ge y}} \alpha_{x+z-y}(a_x) \\ &=
\sum_{\substack{x\in F\\x\ge y}} \alpha_{x+z-y}(a_x) = \alpha_z(b_{-y}),
\end{align*}
since $z\perp y$ implies that $x+z\ge y$ if and only if $x\ge y$. Because $b\in I_H$, we have
\[
\alpha_z(b_{-y})=b_{z-y}\in J_{H-z+y} =
\bigg(\bigcap_{\substack{h\in H\\h\not\le z-y}}\ker \alpha_{(h-z+y)\vee 0}\bigg)^\perp.
\]
However, we also have
\[
\bigcap_{\mathbf{i}\in\supp(y)}\ker \alpha_{\mathbf{i}}\subseteq\bigcap_{\substack{h\in H\\h\not\le z-y}}\ker \alpha_{(h-z+y)\vee 0}.
\]
Indeed, if $h\in H$ with $h-z+y\not \le 0$, then since $h,-z\le 0$,
\[
\emptyset \ne \supp((h-z+y)\vee 0)\subseteq \supp(y),
\]
so $\ker \alpha_{(h-z+y)\vee 0}\supseteq \ker \alpha_\mathbf{i}$ for at least one $\mathbf{i}\in \supp(y)$. Upon taking annihilators, which reverses containment,
\[
\alpha_z(b_{-y})\in \left(\bigcap_{\mathbf{i}\in \supp(y)}\ker \alpha_i\right)^\perp = Q_y^0.
\]
Therefore $b_{-y}\in Q_y$ for all $y\ge 0$. So, $\pi(b)=0$, proving $I=\ker \pi$.
\end{proof}

Proposition \ref{prop:Zn_isomorphism} implies that there is a $\ast$-isomorphism
\[
C_e^\ast(A\nctimes \Z_+^n)=(p_A+I)\left(\frac{B}{I}\rtimes \Z^n\right)(p_A+I)\cong
p_0\left(\tilde{C}\rtimes \Z^n\right)p_0
\]
which fixes the respective completely isometric copies of $A\nctimes P$.

\section{Applications and Examples} \label{sec:applications_examples}

\subsection{Simplicity of the C*-envelope}\label{subsec:simplicity}

In the commutative case, we can give a dynamical characterization of when the C*-envelope of a Nica-covariant semicrossed product is simple. The following definition is standard.

\begin{defn}\label{def:minimal}
A C*-dynamical system $(A,\alpha,P)$ is \textbf{minimal} if $A$ contains no nontrivial $\alpha$-invariant ideals.
\end{defn}

Throughout Section \ref{subsec:simplicity}, let $(G,P)$ be a lattice-ordered abelian group, and let $(A,\alpha,P)$ be a unital C*-dynamical system. Let $(B,\beta,G)$ be the associated product dilation, with inclusion $\iota:A\to B$ and unique maximal $\beta$-invariant $A$-boundary ideal $I$, as in Corollary \ref{cor:envelope_full_corner}. The C*-envelope of $A\nctimes_\alpha P$ is a full corner of $(B/I)\rtimes_{\tilde{\beta}} G$. The following result is an analogue of \cite{davidson_semicrossed_2017}*{Corollary 4.4.4}.

\begin{prop}\label{prop:minimal}
The C*-dynamical system $(A,\alpha,P)$ is minimal if and only if the automorphic C*-dynamical system $(B/I,\tilde{\beta},G)$ is minimal.
\end{prop}

\begin{proof}
Suppose $(A,\alpha,P)$ is minimal. Since $P$ is abelian, for any $p\in P$ the ideal $\ker\alpha_p\triangleleft A$ is $\alpha$-invariant and doesn't contain the unit $1_A$. By minimality, we must have each $\ker\alpha_p=\{0\}$. Therefore the system $(A,\alpha,P)$ is injective. As in Proposition \ref{prop:envelope_injective}, the dilation $B/I$ is a minimal automorphic extension of $A$. By \cite{davidson_semicrossed_2017}*{Proposition 4.4.3}, it follows that $(B/I,\tilde{\beta},G)$ is minimal.

Conversely, suppose $(B/I,\tilde{\beta},G)$ is minimal. Suppose that $J\triangleleft A$ is a nonzero $\alpha$-invariant ideal. Let
\[
K:=
\overline{\bigcup_{g\in G} \beta_g\iota(J)}=
\overline{\sum_{g\in G}\beta_g\iota(J)}
\]
be the $\beta$-invariant ideal in the product dilation $B$ generated by $\iota(J)$. Because $K\cap\iota(A)=\iota(J)$ is nonzero, and $I$ is an $A$-boundary ideal, we must have $K\not\subseteq I$, so the ideal
\[
\frac{K+I}{I}\triangleleft B/I
\]
is nonzero and $\tilde{\beta}$-invariant. By assumption, we must have $(K+I)/I=B/I$, and therefore $K+I=B$. Because $\iota(J)\subseteq K\subseteq K+I$, the injection $\iota$ induces a $\ast$-monomorphism
\[
\frac{A}{J}\to \frac{B}{K+I}\cong\{0\}.
\]
Therefore $A/J\cong \{0\}$, so $J=A$.
\end{proof}

\begin{defn}\label{def:minimal_commutative}
Let $\varphi$ be an action of a semigroup $P$ on a locally compact Hausdorff space $X$. Then $(X,\varphi,P)$ is a \textbf{classical system}. The system $(X,\varphi,P)$ is \textbf{minimal} if $X$ contains no proper nonempty closed $\varphi$-invariant subsets.
\end{defn}

Definitions \ref{def:minimal} and \ref{def:minimal_commutative} are equivalent in the commutative setting $A=C_0(X)$, since ideals correspond to closed subsets. In the classical setting, the following dynamical notion is related to simplicity for crossed products.

\begin{defn}\label{def:top_free}
Let $(X,\varphi,P)$ be a classical system. The action $\varphi$ is \textbf{topologically free} if for any $p,q\in P$ with $p\ne q$, the set $\{x\in X\mid \varphi_p(x)=\varphi_q(x)\}$ has empty interior.
\end{defn}

Now suppose $A=C(X)$ is a commutative unital C*-algebra. Here $X$ is a compact Hausdorff space. Let $(B,\beta,G)$ be the associated product dilation, with inclusion $\iota:A\to B$, and unique maximal $\beta$-invariant $A$-boundary ideal $I$. By Remark \ref{rem:commutative_envelope}, $B$ and $B/I$ are commutative. The C*-dynamical systems $(A,\alpha,P)$ and $(B/I,\tilde{\beta},G)$ arise from classical systems $(C(X),\varphi,P)$ and $(C_0(Y),\psi,G)$ via the usual duality for commutative C*-algebras. The author is grateful to Evgenios Kakariadis and to the referee for suggesting the following variant of \cite{davidson_semicrossed_2017}*{Corollary 4.4.9}.

\begin{prop}\label{prop:simplicity}
With notation as above, the following are equivalent.
\begin{itemize}
\item [(i)] The system $(X,\varphi,P)$ is minimal and $\varphi_p\ne \varphi_q$ for all $p,q\in P$ with $p\ne q$.
\item [(ii)] The system $(Y,\psi,G)$ is minimal and topologically free.
\item [(iii)] The crossed product $C_0(Y)\rtimes_\psi G$ is simple.
\item [(iv)] The C*-envelope $C_e^\ast(C(X)\nctimes_\varphi P)$ is simple.
\end{itemize}
If any of the above hold, then $Y$ is compact, $C(Y)$ is a minimal automorphic extension of $C(X)$, and
\[
C_e^\ast(C(X)\nctimes_\varphi P)\cong
C(Y)\rtimes_\psi G
\]
is a crossed product.
\end{prop}

\begin{proof}
The equivalence of (ii) and (iii) for an amenable group $G$ is a standard result of Archbold and Spielberg \cite{archbold_topologically_1994}*{Corollary}. Because $C_e^\ast(C(X)\nctimes_\varphi P)$ is a full corner of $C(Y)\rtimes_\psi G$, items (iii) and (iv) are equivalent. 

It therefore suffices to prove (i) and (ii) are equivalent. If the system $(A,\alpha,P)$ is injective, then Proposition \ref{prop:envelope_injective} implies that $(B/I,\beta,G)$ is a minimal automorphic extension of $(A,\alpha,P)$, and that the C*-envelope is the associated crossed product. In this case, $C_0(Y)\cong B/I$ is unital, so $Y$ is compact. When $(A,\alpha,P)$ is injective, (i) and (ii) are shown to be equivalent in \cite{davidson_semicrossed_2017}*{Theorem 4.4.8}. However, minimality of $(A,\alpha,P)$ implies that this system is injective as in Proposition \ref{prop:minimal} above. Using Proposition \ref{prop:minimal}, either (i) or (ii) implies $(A,\alpha,P)$ is minimal, so the result follows.
\end{proof}

If Proposition \ref{prop:simplicity} holds, then the simplicity of the C*-envelope implies that any ideal in a C*-cover of $C(X)\nctimes_\varphi P$ is a $(C(X)\nctimes_\varphi P)$-boundary ideal.

\subsection{Direct Limits of Subgroups}

Given a lattice ordered abelian group $(G,P)$, we call $H\subseteq G$ a \textbf{sub-lattice ordered group} of $G$ if $H$ is a subgroup closed under $\vee$ and $\wedge$. (In fact, the identity $g+h=g\vee h+g\wedge h$ shows that it is enough to assume closure under at least one of $\vee$ or $\wedge$.) For any sub-lattice ordered group, $(H,H\cap P)$ is itself a lattice ordered abelian group. Suppose $(A,\alpha,P)$ is a C*-dynamical system, and set $Q:= H\cap P$. By the universal property, there is a natural homomorphism $A\nctimes_{\alpha\vert_Q} Q\to A\nctimes_\alpha P$ induced by the inclusion $P\subseteq Q$. By \cite{davidson_semicrossed_2017}*{Theorem 4.2.9}, the Fock representation is completely isometric on any Nica-covariant semicrossed product. Suppose $A$ acts faithfully on a Hilbert space $K$, then $A\nctimes_{\alpha\vert_Q} Q$ acts faithfully on $K\otimes \ell^2(Q)$. Then the diagram
\[\begin{tikzcd}
    A\nctimes_\alpha P\arrow[r] & B(K\otimes \ell^2(P)) \arrow[d] \\
    A\nctimes_{\alpha\vert_Q} Q \arrow[r] \arrow[u] & B(K\otimes \ell^2(Q))
\end{tikzcd}\]
commutes, where the right-hand map is compression to $K\otimes \ell^2(Q)\subseteq K\otimes \ell^2(P)$. As the bottom map is completely isometric, it follows that the natural map
\[
A\nctimes_{\alpha\vert_Q}Q\to
A\nctimes_{\alpha}P
\]
is completely isometric. Moreover, if $G=\bigcup_{\lambda \in \Lambda} G_\lambda$ is an internal direct limit of sub-lattice ordered groups $G_\lambda\subseteq G$, then it follows that
\[
A\nctimes_\alpha P\cong
\varinjlim_{\lambda \in \Lambda} A\nctimes_{\alpha\vert_{P_\lambda}} P_\lambda,
\]
is a direct limit. Here, $P_\lambda:= G_\lambda \cap P$. Upon identification, we think of
\[
A\nctimes_\alpha P=
\overline{\bigcup_{\lambda \in \Lambda}A\nctimes_{\alpha\vert_{P_\lambda}} P_\lambda}
\]
as an internal direct limit. The next result is that the respective product dilations (Definition \ref{def:product_dilation}) over $P_\lambda$ embed just as nicely.

\begin{prop}\label{prop:product_embed}
Let $(G,P)$ be a lattice ordered abelian group. Let $(A,\alpha,P)$ be a C*-dynamical system, with product dilation $(B,\beta,G)$.
\begin{itemize}
\item[(1)] Suppose $H\subseteq G$ is a sub-lattice ordered group. Setting $Q=H\cap P$, let $(C,\gamma,H)$ be the product dilation for $(A,\alpha\vert_Q,Q)$. Then $C$ embeds into $B$ via an equivariant $\ast$-monomorphism fixing $A$.
\item[(2)] If $G=\bigcup_{\lambda \in \Lambda}G_\lambda$, for sub-lattice ordered groups $G_\lambda$, let $(B_\lambda,\beta_\lambda,G_\lambda)$ be the product dilation for $(A,\alpha_\lambda,P_\lambda)$, where $P_\lambda:= G_\lambda\cap P$ and $\alpha_\lambda :=\alpha\vert_{P_\lambda}$. Then up to identification, we have
\[
B\cong\overline{\bigcup_{\lambda \in \Lambda}B_\lambda}\cong
\varinjlim_{\lambda\in \Lambda} B_\lambda.
\]
\end{itemize}
\end{prop}

\begin{proof}
As in the proof of Proposition \ref{prop:Zn_isomorphism}, there is a well defined $\ast$-homomorphism $\pi:\sum_{g\in G}\beta_g\eta(A)\to B$ with $(\beta_H)_g\eta(a)\mapsto\beta_g\iota(a)$. Here $\iota:A\to B$ and $\eta:A\to C$ are the usual inclusions. Then, (1) follows if we can prove $\pi$ is isometric. Let
\[
b=
\sum_{g\in F}\beta_{-g}\iota(a_g),
\]
where $F\subseteq H$ is finite and $a_g\in A$. Then
\begin{align*}
\|b\|&=
\left\|\sum_{g\in F}\gamma_{-g}\iota_H(a_g)\right\| \\ &=
\sup_{h\in H}\Big\|\sum_{\substack{g\in F\\ g\le h}} \alpha_{h-g}(a_g) \Big\| \\ &\le
\sup_{k\in G}\Big\|\sum_{\substack{g\in F\\ g\le k}} \alpha_{k-g}(a_g) \Big\| = \|\pi(b)\|.
\end{align*}
Conversely, given $k\in G$, since $H$ is $\vee$-closed we have
\[
\{g\in F\mid g\le k\}=
\{g\in F\mid g\le h\},
\]
where $h:=\vee\{g\in F\mid g\le k\}\in H$. Then,
\begin{align*}
\|\pi(b)_k\| &=
\Big\|\sum_{\substack{g\in F \\ g\le k} }\alpha_{k-g}(a_g)\Big\| \\ &=
\Big\|\alpha_{k-h}\Big(\sum_{\substack{g\in F\\ g\le h}}\alpha_{h-g}(a_g)\Big)\Big\| \\ &=
\|\alpha_{k-h}(b_h)\|\le \|b_h\|\le\|b\|.
\end{align*}
So, $\|\pi(b)\|=\|b\|$ and $\pi$ extends to a $\ast$-monomorphism.

For claim (2), it follows from (1) that each $B_\lambda$ embeds in $B$. By minimality,
\[
B=
\overline{\sum_{g\in G}\beta_g\iota(A)} =
\overline{\bigcup_{\lambda \in \Lambda}\sum_{g\in G_\lambda}\beta_g\iota(A)} =
\overline{\bigcup_{\lambda \in \Lambda}B_\lambda},
\]
as claimed.
\end{proof}

Since the embedding in Proposition \ref{prop:product_embed} is equivariant and fixes the copy of $A$, and since all groups involved are abelian and so exact, we also get a $\ast$-embedding
\[
B_\lambda\rtimes_{\beta_\lambda} G_\lambda \subseteq
B\rtimes_\beta G.
\]
This embedding restricts to the natural embedding $A\nctimes P_\lambda \subseteq A\nctimes P$. Moreover
\[
B\rtimes_\beta G\cong
\overline{\bigcup_{\lambda \in \Lambda} B_\lambda \rtimes_{\beta_\lambda}G_\lambda}
\]
is again a direct product. It's then tempting to ask when this result still holds after passing to quotients by Shilov ideals. That is, when is
\[
C_e^\ast(A\nctimes_\alpha P)\cong
\varinjlim_{\lambda \in \Lambda} C_e^\ast(A\nctimes_{\alpha_\lambda}P_\lambda)?
\]
This does occur for surjective systems over totally ordered groups.

\begin{prop}\label{prop:surjective_direct_limit}
Let $(G,P)$ be a totally ordered abelian group, and suppose that $(A,\alpha,P)$ is a unital surjective C*-dynamical system.
\begin{itemize}
\item[(1)] Let $H\subseteq G$ be a subgroup, and set $Q:= H\cap P$. Let $(B,\beta,G)$ (resp. $(C,\gamma,H)$) be the product dilation for $(A,\alpha,P)$ (resp. $(A,\alpha\vert_Q,Q)$). Let $I$ (resp. $J$) be the unique maximal $G$-invariant (resp. $H$-invariant) $A$-boundary ideal in $B$ (resp. $C$). After identifying $C\subseteq B$, we have that
\[
J=
I\cap C.
\]
\item[(2)] Suppose $G=\bigcup_{\lambda \in \Lambda}G_\lambda$ is a directed limit of subgroups. If $(B,\beta,G)$ (respectively $(B_\lambda,\beta_\lambda,G_\lambda)$) is the product dilation for $(A,\alpha,P)$ (resp. $(A,\alpha_\lambda,P_\lambda) = (A,\alpha_\lambda,G_\lambda \cap P)$), and $I\triangleleft B$ and $I_\lambda\triangleleft B_\lambda$ are the respective unique maximal $\beta$ or $\beta_\lambda$-invariant $A$-boundary ideals, then $I_\lambda=I\cap B_\lambda$ and 
\[
I=
\overline{\bigcup_{\lambda\in \Lambda}I_\lambda}.
\]
\end{itemize}
\end{prop}

\begin{proof}
To prove (1), we use Proposition \ref{prop:product_embed} to identify $C\subseteq B$. Since $I$ is a $G$-invariant $A$-boundary ideal, $I\cap C$ is an $H$-invariant boundary ideal in $C$. So, $I\cap B_H\subseteq J$. Conversely, suppose $x\in J$. By Lemma \ref{lem:product_direct_limit}, and inductivity of ideals, it suffices to assume $x$ has the form
\[
x=\sum_{g\in F}\beta_{-g}\iota(a_g) \in I_S^{(H)}:=\{y\in C\mid y_h\in J_{S-h}\triangleleft A \text{ for all }h\in H\}
\]
for some grid $F\subseteq H\subseteq G$, and some finite subset $S\subseteq H$. We will prove
\[
x\in I_S=\{y\in B\mid y_g\in J_{S-g}\text{ for all }g\in G\}\subseteq I.
\]

Let $g\in G$, and
\[
b\in K_{S-g}=
\bigcap_{\substack{s\in S\\ s\not\le g}}\ker \alpha_{(s-g)\vee 0}=
\bigcap_{\substack{s\in S\\ s> g}}\ker \alpha_{s-g}.
\]
The second equality is where we use the assumption that $G$ is totally ordered. As in the proof of Proposition \ref{prop:product_embed}.(1), we find
\[
x_g=
\alpha_{g-h}(x_h),
\]
where
\[
h:=
\bigvee\{k\in F\cup S\mid k\le g\}\in H.
\]
Since the action $\alpha$ is by surjections, we can write $b=\alpha_{g-h}(c)$ for some $c\in A$. Then because $b\in K_{S-g}$, it follows that
\[
c\in 
\bigcap_{\substack{s\in S\\ s> h}}\ker \alpha_{s-h} =K_{S-h}.
\]
Because $x\in I_S^{(H)}$, we conclude $x_hc=0$, so $x_gb=\alpha_{g-h}(x_hc)=0$. Thus $x\in I_S\subseteq I$, as needed.

Claim (2) follows because from (1) and the identification $B=\overline{\bigcup_\lambda B_\lambda}$ (Proposition \ref{prop:product_embed}.(2)), because in this case inductivity of ideals implies
\[
I=
\overline{\bigcup_{\lambda \in \Lambda}I\cap B_\lambda}.
\]
But by (1), $I\cap B_\lambda=I_\lambda$.
\end{proof}

\begin{cor}\label{cor:envelope_direct_limit}
Suppose $(G,P)$ is a totally ordered group with $G=\bigcup_{\lambda \in \Lambda} G_\lambda$, for subgroups $G_\lambda$. If $(A,\alpha,P)$ is a surjective unital C*-dynamical system, then
\[
C_e^\ast(A\nctimes_\alpha P) \cong
\varinjlim_{\lambda \in \Lambda}\;C_e^\ast(A\nctimes_{\alpha\vert_{P_\lambda}}P_\lambda),
\]
where $P_\lambda=G_\lambda\cap P$.
\end{cor}

Corollary \ref{cor:envelope_direct_limit} applies to the totally ordered group $(\Q,\Q_+)$, where we can decompose
\[
\Q=
\bigcup_{n\ge 1}\frac{\Z}{n!}
\]
as a direct limit of an increasing sequence of totally ordered subgroups. More generally, it applies to any subgroup of $\mathbb{R}$ which is built as a union of an increasing sequence of cyclic subgroups, such as the dyadic rationals. It is not clear that one can obtain Corollary \ref{cor:envelope_direct_limit} in vacuo without the explicit description of the Shilov ideal from Theorem \ref{thm:I_is_Shilov_ideal}.

The following examples show that the hypotheses of surjectivity or total ordering of $G$ cannot be dropped from Proposition \ref{prop:surjective_direct_limit}.

\begin{eg}\label{eg:Q_non_surjective}
Define an action $\varphi$ of $\mathbb{R}_+$ on $[-1,1]$ by the continuous maps
\[
\varphi_x(t)=
\begin{cases}
    t, & x=0, \\
    e^{-x}|t|, & x>0.
\end{cases}
\]
Then $\varphi$ is a semigroup action, which is jointly continuous away from $x=0\in \mathbb{R}_+$. This induces an action $\alpha$ of $\mathbb{R}_+$ on $A=C([-1,1])$ by $\ast$-homomorphisms
\[
\alpha_t(f)=
f\circ \varphi_t.
\]
For any $x>0$, $\varphi_x$ is not injective, and so $\alpha_x$ is not surjective. Indeed, for any $f\in C([-1,1])$, $\alpha_x(f)$ is an even function.

Restrict $\alpha$ to get C*-dynamical systems $(A,\alpha,\Z_+)$ and $(A,\alpha,\Z_+/2)$. Build the product dilation $(B,\beta,\Z/2)$ for $(A,\alpha,\Z_+/2)$. By Proposition \ref{prop:product_embed}.(1), we can identify the product dilation for $(A,\alpha,\Z_+)$ as the C*-subalgebra
\[
B_1=\overline{\sum_{n\in \Z} \beta_n\iota(A)}.
\]
We will show that the unique maximal $\beta$-invariant boundary ideal $I_1$ for $A$ in $(B_1,\beta,\Z)$ is not a subset of the unique maximal boundary ideal $I\triangleleft B$ in $(B,\beta,\Z/2)$. By Proposition \ref{prop:envelope_Z}, we have
\[
I=
\left\{x\in B\;\middle|\;  x_{n/2}\in (\ker \alpha_{1/2})^\perp \text{ for all } n\in \Z,\text{ and }\lim_{n\to\infty} x_n=0\right\},
\]
and
\[
I_1=
\left\{x\in B_1\;\middle|\; x_n\in (\ker \alpha_1)^\perp \text{ for all } n\in \Z,\text{ and }\lim_{n\to\infty} x_n=0\right\}.
\]

Suppose that we had $I_1\subseteq I$. Then it would follow that
\[
\alpha_{1/2}\left(\left(\ker\alpha_1\right)^\perp\right)\subseteq (\ker \alpha_{1/2})^\perp.
\]
To prove this, suppose $a\in (\ker \alpha_1)^\perp$. Then
\[
a-
\beta_{-1}\iota(\alpha_1(a))\in I_1,
\]
so by assumption $a-\beta_{-1}\iota(\alpha_1(a))\in I$. Then
\[
[a-\beta_{-1}\iota(\alpha_1(a))]_{1/2}=
\alpha_{1/2}(a)\in (\ker\alpha_{1/2})^\perp.
\]
However, in our case, for $x>0$,
\[
\ker \alpha_x =
\left\{f\in A\mid f\vert_{[0,e^{-x}]}=0\right\}.
\]
So,
\[
\left(\ker \alpha_x\right)^\perp =
C_0((0,e^{-x})) =
\left\{f\in A\mid \supp(f)\subseteq[0,e^{-x}]\right\}.
\]
We certainly cannot have
\[
\alpha_{1/2}\left(C_0((0,e^{-1}))\right)\subseteq
C_0(0,e^{-1/2}),
\]
because $\alpha_{1/2}$ is always an even function and $\alpha_{1/2}\ne 0$. For instance, $f(x)=\max\{x(1-ex),0\}$ satisfies
\[
f\in C_0(0,e^{-1})\quad\text{and}\quad \alpha_{1/2}(f)\not\in C_0(0,e^{-1/2}),
\]
because $\alpha_{1/2}(f)(-e^{-1/2}/2)=f(e^{-1}/2)>0$. So, we cannot have $I_1\subseteq I$ and the conclusion in Proposition \ref{prop:surjective_direct_limit}.(1) fails for the sub-lattice ordered group $\Z\subseteq \Z/2$ when $\alpha$ is not surjective.
\end{eg}

\begin{eg}\label{eg:Z2_not_embed}
Proposition \ref{prop:surjective_direct_limit}.(1) fails in the case $H=\Z\oplus \{0\}\subseteq \Z\oplus \Z=G$, even for surjective actions. Take any C*-dynamical system $(A,\alpha,\Z_+^2)$. Using the same notation as Proposition \ref{prop:surjective_direct_limit}, let $C$ and $B$ be the respective product dilations for $(A,\alpha,\Z_+\oplus \{0\})$ and $(A,\alpha,\Z_+^2)$. Let $J$ and $I$ be the respective unique maximal invariant $A$-boundary ideals in $C$ and $B$. As in Proposition \ref{prop:product_embed}.(1), identify $C\subseteq B$. Then, suppose for a contradiction that $J\subseteq I$.

As $H\cong \Z$, Proposition \ref{prop:envelope_Z} gives
\[
J=
\left\{x\in B\subseteq \prod_{\Z^2} A\; \middle|\; x_{(n,0)}\in (\ker \alpha_1)^\perp \text{ for all }n\in \Z, \text{ and }\lim_{n\to \infty} x_{(n,0)}=0\right\}.
\]
Therefore, if $a\in (\ker \alpha_1)^\perp$, we have
\[
x=\iota(a)-\beta_1^{-1}\iota\alpha_1(a) \in J.
\]
Given $\varepsilon>0$, there is a finite subset $F\subseteq \Z^2$ and an element $y\in I_F$ with $\|x-y\|<\varepsilon$. Since $\{I_F\mid F\subseteq G\text{ finite}\}$ is directed, we are free to enlarge $F$ so that $(1,1)\in F$. Set
\[
k=
\max\{m\mid (n,m)\in F\}.
\]
Then for $j\ge k$, we have
\begin{align*}
y_{(0,j)} &\in
\bigg(\bigcap_{\substack{(n,m)\in F\\(n,m-j)\not \le 0}}\ker \alpha_{(n,m-j)\vee 0}\bigg)^\perp\\ &=
\bigg(\bigcap_{\substack{(n,m)\in F\\ n> 0}}\ker \alpha_1^n\bigg)^\perp \subseteq 
\left(\ker \alpha_1\right)^\perp,
\end{align*}
so
\[
\dist(\alpha_2^j(a),(\ker \alpha_1)^\perp) \le
\|x-y\|<\varepsilon.
\]
This proves that for any commuting unital endomorphisms $\alpha_1,\alpha_2\in \End(A)$, and any $a\in (\ker \alpha_1)^\perp$, that
\begin{equation}\label{eq:almost_invariant_perp}
\lim_{j\to\infty}\dist(\alpha_2^j(a),(\ker \alpha_1)^\perp)=
\lim_{j\to \infty}\|\alpha_2^j(a)+(\ker \alpha_1)^\perp\|=0.
\end{equation}

However, the identity \eqref{eq:almost_invariant_perp} fails in general. Let $X=[0,1]\times [0,1]$ and $A=C(X)$. The two injective continuous maps $\varphi_1,\varphi_2:X\to X$ defined by
\[
\varphi_1(s,t)=\left(\frac{s}{2},t\right),\qquad
\varphi_2(s,t)=\left(\frac{s}{2},\frac{t}{2}\right)
\]
commute and define surjective $\ast$-endomorphisms $\alpha_i\in \End(A)$, where $\alpha_i(f)=f\circ \varphi_i$, for $i=1,2$. Then
\begin{align*}
(\ker \alpha_1)^\perp &=
C_0([0,1/2)\times [0,1]), \text{ and}\\
(\ker \alpha_2)^\perp &=
C_0([0,1/2)\times [0,1/2)).
\end{align*}
Pick any $f\in (\ker \alpha_1)^\perp$ with $f(s,t)=1$ whenever $s\in [0,3/8]$. Then we have $\alpha_2^j(f)(3/4,0)=1$ for any $j\ge 1$. So,
\[
\|\alpha_2^j(f)+(\ker \alpha_1)^\perp\| =
\left\|\alpha_2^j(f)\vert_{[1/2,1]\times [0,1]}\right\| \ge 1
\]
for all $j$, and \eqref{eq:almost_invariant_perp} does not hold. We conclude that Proposition \ref{prop:surjective_direct_limit}.(1) fails for the surjective system $(A,\alpha,\Z^2_+)=(C(X),\alpha,\Z^2_+)$, with the sub-lattice ordered group $H=\Z\oplus \{0\}\subseteq \Z^2$.
\end{eg}


\begin{thebibliography}{99}
\bib{archbold_topologically_1994}{article}{
  title={Topologically free actions and ideals in discrete C*-dynamical systems},
  author={Archbold, Robert J and Spielberg, Jack S},
  journal={Proceedings of the Edinburgh Mathematical Society},
  volume={37},
  pages={119--124},
  year={1994},
  publisher={Cambridge University Press}
}

\bib{arveson_operator_1967}{article}{
  author  = {Arveson, William B.},
  title   = {Operator algebras and measure preserving automorphisms},
  journal = {Acta Mathematica},
  year    = {1967},
  volume  = {118},
  pages   = {95--109},
}

\bib{arveson_subalgebras_1969}{article}{
  author  = {Arveson, William B.},
  title   = {Subalgebras of {C}*-algebras},
  journal = {Acta Mathematica},
  year    = {1969},
  volume  = {123},
  pages   = {141--224},
}

\bib{arveson_operator_1969}{article}{
  author  = {Arveson, William B.},
  author  = {Josephson, Keith B.},
  title   = {Operator algebras and measure preserving automorphisms {II}},
  journal = {Journal of Functional Analysis},
  year    = {1969},
  volume  = {4},
  pages   = {100--134}
}

\bib{brown_c*-algebras_2008}{book}{
  title     = {C*-{Algebras} and {Finite}-{Dimensional} {Approximations}},
  publisher = {American Mathematical Soc.},
  year      = {2008},
  author    = {Brown, Nathanial Patrick},
  author    = {Ozawa, Narutaka},
}

\bib{carlsen_co-universal_2011}{article}{
  author  = {Carlsen, Toke M.},
  author    = {Larsen, Nadia S.},
  author    = {Sims, Aidan},
  author    = {Vittadello, Sean T.},
  title   = {Co-universal algebras associated to product systems, and gauge-invariant uniqueness theorems},
  journal = {Proceedings of the London Mathematical Society},
  year    = {2011},
  volume  = {103},
  pages   = {563--600}
}

\bib{davidson_semicrossed_2017}{article}{
  author  = {Davidson, Kenneth R.},
  author  = {Fuller, Adam},
  author  = {Kakariadis, Evgenios},
  title   = {Semicrossed products of operator algebras by semigroups},
  journal = {Memoirs of the American Mathematical Society},
  year    = {2017},
  volume  = {247},
  number  = {1168}
}

\bib{davidson_semicrossed_2018}{article}{
  author  = {Davidson, Kenneth R},
  author    = {Fuller, Adam H},
  author=  {Kakariadis, Evgenios T A},
  title   = {Semicrossed products of operator algebras: a survey},
  journal = {New York J. Math},
  year    = {2018},
  volume  = {24},
  pages   = {56--86},
}

\bib{davidson_conjugate_2014}{article}{
  author  = {Davidson, Kenneth R.},
  author = {Kakariadis, Evgenios T. A.},
  title   = {Conjugate {Dynamical} {Systems} on {C}*-algebras},
  journal = {International Mathematics Research Notices},
  year    = {2014},
  volume  = {2014},
  pages   = {1289--1311}
}

\bib{davidson_dilating_2010}{article}{
  author  = {Davidson, Kenneth R.},
  author = {Katsoulis, Elias G.},
  title   = {Dilating covariant representations of the non-commutative disc algebras},
  journal = {Journal of Functional Analysis},
  year    = {2010},
  volume  = {259},
  pages   = {817--831}
}

\bib{davidson_isomorphisms_2008}{article}{
  author  = {Davidson, Kenneth R.},
  author = {Katsoulis, Elias G.},  
  title={Isomorphisms between topological conjugacy algebras},
  journal={Journal f{\"u}r die reine und angewandte Mathematik (Crelles Journal)},
  volume={2008},
  number={621},
  pages={29--51},
  year={2008},
}

\bib{davidson_operator_2011}{book}{
  title     = {Operator {Algebras} for {Multivariable} {Dynamics}},
  publisher = {American Mathematical Soc.},
  year      = {2011},
  author    = {Davidson, Kenneth R.},
  author    = {Katsoulis, Elias G.},
}

\bib{davidson_noncommutative_1998}{article}{
  author  = {Davidson, Kenneth R.},
  author    = {Popescu, Gelu},
  title   = {Noncommutative {Disc} {Algebras} for {Semigroups}},
  journal = {Canadian Journal of Mathematics},
  year    = {1998},
  volume  = {50},
  pages   = {290--311}
}

\bib{davidson_c*-envelopes_2008}{article}{
  author  = {Davidson, Kenneth R.},
  author = {Roydor, Jean},
  title   = {C*-envelopes of tensor algebras for multivariable dynamics},
  journal={Proceedings of the Edinburgh Mathematical Society},
  volume={53},
  pages={333--351},
  year={2010},
  note={\emph{Contains an error, corrected in corrigendum (2011) or in \cite{kakariadis_contributions_2012}.}}
}

\bib{dinh_discrete_1991}{article}{
  author  = {Dinh, Hung T.},
  title   = {Discrete product systems and their {C}*-algebras},
  journal = {Journal of Functional Analysis},
  year    = {1991},
  volume  = {102},
  pages   = {1--34}
}

\bib{dor-on_tensor_2018}{article}{
  author  = {Dor-On, Adam},
  author = {Katsoulis, Elias},
  title   = {Tensor algebras of product systems and their {C}*-envelopes},
  journal={Journal of Functional Analysis},
  pages={108416},
  year={2019},
}

\bib{fowler_compactly-aligned_1998}{article}{
  author  = {Fowler, Neal J.},
  title   = {Compactly-aligned discrete product systems, and generalizations of $\mathcal{O}_\infty$},
  journal={International Journal of Mathematics},
  volume={10},
  pages={721--738},
  year={1999},
}

\bib{fowler_discrete_2002}{article}{
  author  = {Fowler, Neal J.},
  title   = {Discrete product systems of {Hilbert} bimodules},
  journal = {Pacific Journal of Mathematics},
  year    = {2002},
  volume  = {204},
  pages   = {335--375}
}

\bib{fowler_representations_2003}{article}{
  author  = {Fowler, Neal J.},
  author = {Muhly, Paul S.},
  author = {Raeburn, Iain},
  title   = {Representations of {Cuntz}-{Pimsner} algebras},
  journal = {Indiana University Mathematics Journal},
  year    = {2003},
  volume  = {52},
  pages   = {569--606},
}

\bib{fowler_discrete_1998}{article}{
  author  = {Fowler, Neal J.},
  author = {Raeburn, Iain},
  title   = {Discrete {Product} {Systems} and {Twisted} {Crossed} {Products} by {Semigroups}},
  journal = {Journal of Functional Analysis},
  year    = {1998},
  volume  = {155},
  pages   = {171--204}
}

\bib{hadwin_operator_1988}{article}{
  author  = {Hadwin, Donald W.},
  author = {Hoover, Thomas B.},
  title   = {Operator algebras and the conjugacy of transformations},
  journal = {Journal of Functional Analysis},
  year    = {1988},
  volume  = {77},
  pages   = {112--122}
}

\bib{hamana_injective_1979}{article}{
  title={Injective envelopes of operator systems},
  author={Hamana, Masamichi},
  journal={Publications of the Research Institute for Mathematical Sciences},
  volume={15},
  number={3},
  pages={773--785},
  year={1979},
  publisher={Research Institute forMathematical Sciences}
}


\bib{huef_nuclearity_2019}{article}{
  author  = {Huef, Astrid an},
  author = {Nucinkis, Brita},
  author = {Sehnem, Camila F.},
  author = {Yang, Dilian},
  title   = {Nuclearity of semigroup {C}*-algebras},
  journal = {arXiv preprint arXiv:1910.04898},
  year    = {2019}
}

\bib{kakariadis_semicrossed_2011}{article}{
  author  = {Kakariadis, Evgenios T. A.},
  title   = {Semicrossed products of {C}*-algebras and their {C}*-envelopes},
  journal = {arXiv preprint arXiv:1102.2252},
  year    = {2011}
}

\bib{kakariadis_representations_2014}{article}{
  author  = {Kakariadis, Evgenios T. A.},
  author = {Peters, Justin R.},
  title   = {Representations of {C}*-dynamical systems implemented by {Cuntz} families},
  journal = {Muenster Journal of Mathematics},
  volume = {6},
  pages = {383-411},
  year    = {2013}
}

\bib{kakariadis_contributions_2012}{article}{
  title={Contributions to the theory of C*-correspondences with applications to multivariable dynamics},
  author  = {Kakariadis, Evgenios T. A.},
  author = {Katsoulis, Elias G.},
  journal={Transactions of the American Mathematical Society},
  volume={364},
  number={12},
  pages={6605--6630},
  year={2012}
}

\bib{kakariadis_semicrossed_2012}{article}{
  author  = {Kakariadis, Evgenios T. A.},
  author = {Katsoulis, Elias G.},
  title   = {Semicrossed products of operator algebras and their {C}*-envelopes},
  journal = {Journal of Functional Analysis},
  year    = {2012},
  volume  = {262},
  pages   = {3108--3124}
}

\bib{katsoulis_tensor_2006}{article}{
  author  = {Katsoulis, Elias G.},
  author = {Kribs, David W.},
  title   = {Tensor algebras of {C}*-correspondences and their {C}*-envelopes},
  journal = {Journal of Functional Analysis},
  year    = {2006},
  volume  = {234},
  pages   = {226--233}
}

\bib{laca_endomorphisms_1999}{article}{
  author     = {Laca, Marcelo},
  title      = {From endomorphisms to automorphisms and back: dilations and full corners},
  journal={Journal of the London Mathematical Society},
  volume={61},
  pages={893--904},
  year={2000}
}

\bib{laca_semigroup_1996}{article}{
  author  = {Laca, Marcelo},
  author = {Raeburn, Iain},
  title   = {Semigroup {Crossed} {Products} and the {Toeplitz} {Algebras} of {Nonabelian} {Groups}},
  journal = {Journal of Functional Analysis},
  year    = {1996},
  volume  = {139},
  pages   = {415--440}
}

\bib{li_regular_2017}{article}{
  author  = {Li, Boyu},
  title   = {Regular {Dilation} and {Nica}-covariant {Representation} on {Right} {LCM} {Semigroups}},
  journal={Integral Equations and Operator Theory},
  volume={91},
  pages={36},
  year={2019},
}

\bib{li_regular_2016}{article}{
  author  = {Li, Boyu},
  title   = {Regular representations of lattice ordered semigroups},
  journal = {Journal of Operator Theory},
  year    = {2016},
  volume  = {76},
  pages   = {33--56}
}

\bib{li_semigroup_2012}{article}{
  author  = {Li, Xin},
  title   = {Semigroup {C}*-algebras and amenability of semigroups},
  journal = {Journal of Functional Analysis},
  year    = {2012},
  volume  = {262},
  pages   = {4302--4340}
}

\bib{muhly_tensor_1998}{article}{
  author     = {Muhly, Paul S.},
  author = {Solel, Baruch},
  title      = {Tensor {Algebras} over {C}*-{Correspondences}: {Representations}, {Dilations}, and {C}*-{Envelopes}},
  journal    = {Journal of Functional Analysis},
  year       = {1998},
  volume     = {158},
  pages      = {389--457}
}

\bib{murphy_crossed_1994}{article}{
  author  = {Murphy, Gerard J.},
  title   = {Crossed {Products} of {C}*-{Algebras} by {Semigroups} of {Automorphisms}},
  journal={Proceedings of the London Mathematical Society},
  volume={3},
  pages={423--448},
  year={1994},
}

\bib{nica_c*-algebras_1992}{article}{
  author  = {Nica, Alexandru},
  title   = {C*-algebras generated by isometries and {Weiner}-{Hopf} operators},
  journal = {Journal of Operator Theory},
  year    = {1992},
  volume  = {27},
  pages   = {17--52},
}

\bib{paulsen_completely_2003}{book}{
  title     = {Completely {Bounded} {Maps} and {Operator} {Algebras}},
  publisher = {Cambridge University Press},
  year      = {2003},
  author    = {Paulsen, Vern},
  edition   = {1}
}

\bib{peters_semi-crossed_1984}{article}{
  author  = {Peters, Justin},
  title   = {Semi-crossed products of {C}*-algebras},
  journal = {Journal of Functional Analysis},
  year    = {1984},
  volume  = {59},
  pages   = {498--534}
}

\bib{sehnem_c*-algebras_2019}{article}{
  author  = {Sehnem, Camila F.},
  title   = {On {C}*-algebras associated to product systems},
  journal = {Journal of Functional Analysis},
  year    = {2019},
  volume  = {277},
  pages   = {558--593}
}

\bib{sims_c*-algebras_2007}{article}{
  author  = {Sims, Aidan},
  author = {Yeend, Trent},
  title   = {C*-algebras associated to product systems of {Hilbert} bimodules},
  journal={Journal of Operator Theory},
  pages={349--376},
  year={2010},
  publisher={JSTOR}
}

\bib{zahmatkesh_nica-toeplitz_2019}{article}{
  author  = {Zahmatkesh, Saeid},
  title   = {The {Nica}-{Toeplitz} algebra of abelian lattice ordered groups is a full corner in group crossed product},
  journal = {arXiv preprint arXiv:1912.09682},
  year    = {2019}
}

\bib{zahmatkesh_partial-isometric_2017}{article}{
  author  = {Zahmatkesh, Saeid},
  title   = {The partial-isometric crossed products by semigroups of endomorphisms are {Morita} equivalent to crossed products by groups},
  journal = {arXiv preprint arXiv:1710.06708},
  year    = {2017}
}

\end{thebibliography}
\end{document}